\documentclass[12pt,reqno]{amsart}
\usepackage{amsfonts, amssymb, amsthm, amsmath}
\allowdisplaybreaks[4]
\usepackage{rotating}
\usepackage{tikz}
\usepackage{graphics}
\usepackage{amscd}
\usepackage[latin2]{inputenc}
\usepackage{t1enc}
\usepackage[mathscr]{eucal}
\usepackage{extarrows}
\usepackage{indentfirst}
\usepackage{graphicx}
\usepackage{pict2e}
\usepackage{mathrsfs}
\usepackage{enumerate}
\usepackage{CJK}
\usepackage{cite}
\usepackage{color}
\usepackage{epic}
\theoremstyle{plain}
\newtheorem{thm}{Theorem}%[section]

\newtheorem{lem}[thm]{Lemma}

\theoremstyle{definition}

\numberwithin{equation}{section}
\newtheorem*{rmk}{Remark}
\newtheorem*{ex}{Example}

\def\N{\mathbb{N}}

\def\L{\mathcal{LH}}
\DeclareMathOperator\exc{exc}

\DeclareMathOperator\cyc{cyc}

\DeclareMathOperator\erec{erec}
\DeclareMathOperator\arec{arec}

\DeclareMathOperator\cpeak{cpeak}

\DeclareMathOperator\cdrise{cdrise}
\DeclareMathOperator\cdfall{cdfall}

\DeclareMathOperator\cval{cval}
\DeclareMathOperator\fix{fix}

\def\upnest{\textrm{unest}}
\def\lownest{\textrm{lnest}}
\mathchardef\mhyphen="2D
\DeclareMathOperator\les{(31\mhyphen 2)}
\DeclareMathOperator\less{(13\mhyphen 2)}
\DeclareMathOperator\res{(2\mhyphen 13)}
\DeclareMathOperator\ress{(2\mhyphen 31)}

\DeclareMathOperator\rar{rar}

 \DeclareMathOperator\Exc{Exc}
  
 \DeclareMathOperator\Excp{Excp}
  \DeclareMathOperator\Erecp{Erecp}
    \DeclareMathOperator\Erecl{Erecl}
 \DeclareMathOperator\Excl{Excl}

\def\312{\les}
\def\132{\less}
\def\213{\res}
\def\231{\ress}
\DeclareMathOperator\rec{rec}

\DeclareMathOperator\Recp{Recp}
\DeclareMathOperator\Arecp{Arecp}
\DeclareMathOperator\Arec{Arec}
\DeclareMathOperator\Rec{Rec}
\DeclareMathOperator\Erec{Erec}
\DeclareMathOperator\Recl{Recl}
\DeclareMathOperator\Arecl{Arecl}
\DeclareMathOperator\Cyc{Cyc}
\DeclareMathOperator\Rar{Rar}
\DeclareMathOperator\Rarp{Rarp}
\DeclareMathOperator\Rarl{Rarl}
\DeclareMathOperator\Cval{Cval}
\DeclareMathOperator\Cdfall{Cdfall}
\DeclareMathOperator\Cdrise{Cdrise}
\DeclareMathOperator\Cpeak{Cpeak}
\DeclareMathOperator\Fix{Fix}

\def\N{\mathbb{N}}
\def\SS{\mathfrak{S}}

\def\su{\textsc{U}}
\def\sd{\textsc{D}}
\def\la{\textsc{L}_a}
\def\lb{ \textsc{L}_b}
\def\lc{ \textsc{L}_c}

\title[Further  equidistribution of set-valued statistics]{Further equidistribution of set-valued statistics  on permutations}
\author{Jianxi Mao}
\address[Jianxi Mao]{School of Mathematic Sciences, 
 Dalian University of Technology, Dalian 116024, P. R. China}
\email{maojianxi@hotmail.com}

\author[Jiang Zeng]{Jiang Zeng}
\address[Jiang Zeng]{Univ Lyon, Universit\'e Claude Bernard Lyon 1, CNRS UMR 5208, Institut Camille Jordan, 43 blvd. du 11 novembre 1918, F-69622 Villeurbanne cedex, France}
\email{zeng@math.univ-lyon1.fr}
\date{\today}

\begin{document}
\maketitle
\begin{abstract}
We construct  bijections to show that two pairs of sextuple set-valued 
statistics of permutations are equidistributed on symmetric groups.
This extends 
a recent result of Sokal and the second author valid for integer-valued statistics as well as 
a previous result of Foata and Han for  bivariable  set-valued statistics.
\end{abstract}

\section{Introduction}
Equidistribution problems of set-valued statistics on permutations have attracted much attention in recent literature, see \cite{BV17, KL18, PO14, FH09}.
A decade ago, answering a conjecture of Foata and Han,   
Cori~\cite{CO09} proved that
 the number of permutations in the symmetric group $\SS_n$  with $p$ cycles and $q$ left-to-right maxima is equal to the number of permutations in  ${\SS}_n$ with $q$ cycles and $p$ left-to-right maxima. In a follow-up~\cite{FH09} Foata and Han  showed that Cori's result 
 can be further extended to set-valued statistics by using two simple  permutation codings called the $A$-code and the $B$-code.
 %see also 
%\cite{CGG13, ELW15, PE09, PO14} for related results inspired from the latter.
 Recently Sokal and the second author~\cite{SZ19} have extended Cori's result on integer-valued 
 bi-statistics    to integer-valued  multiple 
 statistics. The purpose of this paper is to show that the latter has also a set-valued analogue using  a classical \emph{Laguerre history} encoding of permutations~\cite{DV94,CSZ97} and two new bijections from 
 Laguerre histories onto themselves.
For a permutation $\sigma=\sigma(1)\cdots \sigma(n)$ of
$12\ldots n$,
 the pair $(i, \sigma(i))$ ($1\leq i\leq n$)
  is called 
 \begin{itemize}
\item  a \emph{record}  of $\sigma$ 
 if $ \sigma (j) < \sigma (i)$ for all $ j < i$;  
\item  an \emph{antirecord}  of $\sigma$
 if $ \sigma (j) > \sigma (i)$ for all $ j > i$; 
\item an \emph{exclusive record}  of $\sigma$
 if it is a record but not an antirecord; 
\item a \emph{record-antirecord}  of $\sigma$
 if it is both a record and an antirecord;
\item  an \emph{excedance} 
 if $ \sigma(i) > i$.
\end{itemize}
 The corresponding numbers of the above five 
 statistics   are denoted by 
  $\rec\sigma$, $\arec\sigma$, $\erec\sigma$, $\rar\sigma$
 and $\exc\sigma$. Moreover, the indices $i$ and $\sigma(i)$ are called 
 \emph{position} and \emph{letter} of the corresponding statistic. 
 The ten sets of corresponding positions and letters  of the above five 
  statistics are denoted, respectively, by 
$\Recp\sigma$, $\Recl\sigma$, 
$\Arecp\sigma$, $\Arecl\sigma$,
$\Erecp\sigma$,   $\Erecl\sigma$, 
$\Rarp\sigma$, $\Rarl \sigma$,
$\Excp\sigma$,   and $\Excl\sigma$.  For convenience we intoduce 
the four  bi-set-valued statistics    
\begin{align*}
\Rec\sigma&=(\Recp\sigma, \Recl\sigma),\quad 
\Arec\sigma=(\Arecp\sigma, \Arecl\sigma)\\
\Erec\sigma&=(\Erecp\sigma, \Erecl\sigma),\quad 
%\Rar\sigma=(\Rarp\sigma, \Rarl \sigma)\\
\Exc\sigma=(\Excp\sigma, \Excl\sigma).
\end{align*}
Since  the position and letter of a record-antirecord must 
be equal, we have $\Rar\sigma:=\Rarp\sigma=\Rarl \sigma$.
Furthermore,  if the bijection $i\mapsto \sigma(i)$ ($1\leq i\leq n$) has 
 $r$ disjoint cycles, whose \emph{maximum elements} are $c_1, \ldots, c_r$, we set
 $\Cyc\sigma:=\{c_1, c_2, \ldots, c_r\}$,
 and an index $i\in [n]:=\{1, \ldots, n\}$ is called a
 \begin{itemize}
\item \emph{cycle peak}  of $\sigma$ if $ \sigma ^{-1}(i) < i >\sigma (i); $ 
\item  \emph{cycle valley}  of $\sigma$   if $ \sigma ^{-1}(i) > i < \sigma (i); $ 
\item \emph{cycle double rise}  of $\sigma$  if $ \sigma ^{-1}(i) < i < \sigma (i);$ 
\item \emph{cycle double fall}  of $\sigma$  if $ \sigma ^{-1}(i) > i > \sigma (i);$
\item  \emph{fixed point}  of $\sigma$  if $ \sigma (i)=i$.
\end{itemize}
The corresponding sets (resp. numbers) of  the above statistics are denoted  by 
$\Cpeak\sigma$, $\Cval\sigma$, $\Cdrise\sigma$, $\Cdfall\sigma$ and 
$\Fix\sigma$ (resp. 
$\cpeak\sigma$, $\cval\sigma,\, \cdrise\sigma, \,\cdfall\sigma$ and $\fix\sigma$), respectively.  The following is our running example in this paper.

\begin{ex}
Consider the permutation $\sigma\in\SS_{17}$ given by
$$\sigma=\left(%
\begin{array}{cccccccccccccccccc}
1 & 2 & 3 & 4 & 5 & 6 & 7 & 8 & 9 & 10 & 11 & 12 & 13 & 14 & 15 & 16 & 17 \\
4 & 9 & 2 & 11 & 5 & 10 & 1 & 3 & 6 & 8 & 7 & 12 & 16 & 17 & 13 & 14 & 15  \\
\end{array}%
\right).$$
Then
\begin{align*}
\Rec\sigma&=(\{1, \, 2, \,4, \,12, \,13, \,14\}, \, \{4, \,9, \,11, \,12, \,16, \,17\})\\
\Arec\sigma&=(\{7, \,8,  \,9, \,11, \,12, \,15, \,16, \,17\},\,
\{1, \,3, \,6, \,7, \,12, \,13, \,14, \,15\})\\
\Erec\sigma&=(\{1, \,2, \,4, \,13, \,14\},\,\{4, \,9, \,11, \,16, \,17\})\\
\Exc\sigma&=(\{1, \,2, \,4, \,6, \,13, \,14\},\, \{4, \,9, \,10, \,11, \,16, \,17\})
\end{align*}
and $\Rar\sigma=\{12\}$.
%$\Erecp\sigma=\{1,2,4,13,14\}$,
%$\Erecl\sigma=\{4,9,11,16,17\}$,
The cycle decomposition of $\sigma$   reads
\begin{equation*}
\sigma=(1,4,11,7)\,(2,9,6,10,8,3)\,(5)\,(12)\,(13,16,14,17,15),
\end{equation*}
thus
$\Cyc\sigma=\{5,10,11,12,17\}$, 
$\Cpeak\sigma=\{9,10,11,16,17\}$,
$\Cval\sigma=\{1,2,6,13,14\}$,
$\Cdrise\sigma=\{4\}$,
$\Cdfall\sigma=\{3,7,8,15\}$,  and 
$\Fix\sigma=\{5,12\}$.  
\end{ex}  

In a recent work~\cite{SZ19}   Sokal and the second author
studied the  polynomials 
$\mu_n(x,y,u,v)$ in the Taylor expansion of  the Stieltjes continued fraction 
\begin{align}
\sum_{n\geq 0}\mu_n(x,y,u,v)t^n
=\cfrac{1}{1-\cfrac{x\,t}
{1-\cfrac{y\, t}
{1-\cfrac{(x+u)\,t}{1-\cfrac{(y+v)\,t}{\cdots}}}}},
\end{align}
where the  coefficients $\alpha_k$ ($k\geq 1$)  are defined by
\begin{align*}%\label{coefficients-alpha}
\alpha_{2k-1}=x+(k-1)u,\quad
\alpha_{2k}=y+(k-1)v.
\end{align*}
They~\cite[Theorem 2.2]{SZ19}  showed that  the  polynomial $\mu_n$ has 
the two interpretations
\begin{align}
\mu_n(x,y,u,v)&=
\sum_{\sigma\in \SS_n} x^{\arec(\sigma)}
 y^{\erec(\sigma)}u^{n-\exc(\sigma)-\arec(\sigma)}v^{\exc(\sigma)-\erec(\sigma)}\label{eq:arec}\\
 &=\sum_{\sigma\in \SS_n} x^{\cyc(\sigma)}
 y^{\erec(\sigma)}u^{n-\exc(\sigma)-\cyc(\sigma)}v^{\exc(\sigma)-\erec(\sigma)},
 \label{eq:cyc}
\end{align}
where %$\SS_n$ denotes the set of permutations  of $[n]:=\set{1,\dots,n}$ and
$\cyc\sigma$ is the cycle number of $\sigma$.
%For a permutation $\sigma=\sigma(1)\cdots \sigma(n)$ of
%$12\ldots n$,
%.
Besides,
they~\cite[Theorem 2.4]{SZ19}   proved 
the Jacobi continued fraction  expansion
\begin{align}\label{cyc-arec-sym}
&\sum_{n=0}^\infty \left(\sum_{\sigma\in \SS_n} x^{\cyc(\sigma)} y^{\arec(\sigma)} z^{\exc(\sigma)} w_0^{\rar\sigma} \right)t^n=\\
&\cfrac{1}{1-xyw_0\, t-\cfrac{xyz\,t}{1-( x+y+z)\,t-\cfrac{(x+1)(y+1)z\;t}{1-\cdots}}},\nonumber
\end{align}
where  the coefficients $\gamma_k$ ($k\geq 0$)
and $\beta_k$ ($k\geq 1$) are defined by $\gamma_0=xyw_0$,
\begin{align*}
%\gamma_0&=xyw_0\\
\gamma_n=x+y +n-1+nz,\quad 
\beta_n=(x+n-1) (y+n-1)  \quad \textrm{for}\quad  n\geq 1.
\end{align*}
It follows from \eqref{eq:arec} and \eqref{eq:cyc} that the two triple integer-valued statistics 
\begin{align}\label{erec-cyc-exc}
(\erec, \cyc, \exc)\quad  \textrm{and}\quad  (\erec, \arec,\exc)
\end{align}
are equidistributed on $\SS_n$. As 
the right-hand side of \eqref{cyc-arec-sym}  is symmetric under $x\leftrightarrow y$,  the two 
triple integer-valued statistics 
\begin{align}\label{cyc-arec-exc}
( \cyc, \arec, \exc)\quad  \textrm{and}\quad  (\arec, \cyc,\exc)
\end{align} 
are also equidistributed on $\SS_n$.
%This raises the natural question to find a  direct combinatorial interpretation of these two equidistributions of statistics. 
Note that  the symmetry of  the integer valued statistics 
$(\arec, \cyc)$ is  due to Cori~\cite{CO09}. 
Motivated by the set-valued analogue of Cori's result~\cite{FH09}, we shall prove the following 
 set-valued analogue of \eqref{erec-cyc-exc} and \eqref{cyc-arec-exc}.
 \begin{thm}
The two sextuple set-valued statistics 
$$(\Cyc, \Erec, \Exc, \Rar)\quad \text{and}\quad 
(\Arecp, \Erec, \Exc, \Rar)
$$ 
are equidistibuted on $\SS_n$.
	\end{thm}                                                         

\begin{thm}
The two quintuple set-valued statistics 
$$(\Cyc, \Arecp, \Exc, \Rar)\quad \text{and}\quad
(\Arecp, \Cyc, \Exc,\Rar)$$
 are equidistributed on $\SS_n$.                                                  
\end{thm}
As $\Rec=(\Erecp\cup\Rar,\Erecl\cup\Rar) $, 
we derive immediately that the following two pairs of triple set-statistics 
\begin{align}
	(\Cyc, \Rec,\Exc)\quad \text{and}\quad
	(\Arecp,  \Rec,\Exc)\label{han-strong},\\
	(\Cyc, \Arecp,\Exc)\quad \text{and}\quad
	(\Arecp,  \Cyc,\Exc)
\end{align}
are equidistributed on $\SS_n$.	  We note  that \eqref{han-strong} is a stronger version 
of Han's conjecture~\cite{Han19}, i.e., 
$(\cyc,  \Recp,\Excp)$ and  $(\arec,  \Recp, \Excp)$
are equidistributed on $\SS_n$ and  the bijections in \cite{CO09,FH09} do not 
keep track of excedances.
%To prove  Theorems~1 and 2  we shall first encode permutations by Laguerre histories and then construct two bijections $\rho_1$ and $\rho_2$ on Laguerre histories. 

A 3-\emph{Motzkin word} of length $n$   is a word $s:=s_1\ldots s_i$ on the alphabet $\{\su, \sd, \la, \lb, \lc\}$
such that  $|s_1\ldots s_n|_\su= |s_1\ldots s_n|_\sd$ and 
\begin{equation}\label{heigt-condition}
  h_{i}:=|s_1\ldots s_i|_\su-|s_1\ldots s_i|_\sd\geq 0 \quad (i=1, \ldots, n),
\end{equation}
where $|s_1\ldots s_i|_{ \textbf{a}}$ is the number occurences of letter \textbf{a} in the word $s_1\ldots s_i$.   
A \emph{Motzkin path} of length $n$ is a sequence of points 
$\omega=(\omega_0,\omega_1,\ldots,\omega_{n-1},\omega_{n})$ 
in $\N\times \N$, 
starting from $\omega_0=(0,0)$ and  
ending at $\omega_{n}=(n,0)$, such that each step $s_i:=(\omega_{i-1}, \omega_i)$ is of type  \emph{up} if  
$\omega_i-\omega_{i-1}=(1,1)$, or type \emph{down} if  $\omega_i-\omega_{i-1}=(1,-1)$, or type
\emph{level} if $\omega_i-\omega_{i-1}= (1,0)$ for $i=1, \ldots, n$.
Clearly we can depict a Motzkin word  with a Motzkin path 
 with three types of level steps.  A \emph{Laguerre history} of length  $n$ 
is  a pair $(s, \gamma)$, where $s$ is 
a 3-Motzkin word $s:=s_1\ldots s_n$ with $h_n=0$ (i.e., $|s]_{\su}=|s]_{\sd}$)
and $\gamma=(\gamma_1, \ldots, \gamma_n)$ is a sequence satisfying the following:  
\begin{itemize}
\item $\gamma_i=(\xi_i, \eta_i)=(\Delta, \Delta)$ if $s_i=\su$,
\item $\gamma_i=(\xi_i, \eta_i)\in \{1,\ldots, h_{i-1}\}^2$ if $s_i=\sd$, 
\item $\gamma_i=(\xi_i, \eta_i)\in \{1,\ldots, h_{i-1}\}\times \{\Delta\}$ if $s_i=\la$,
\item $\gamma_i=(\xi_i, \eta_i)\in  \{\Delta\}\times \{1,\ldots, h_{i-1}\}$ if $s_i=\lb$,
\item $\gamma_i=(\xi_i, \eta_i)=(\Delta, h_{i-1}+1)$  if $s_i=\lc$.
\end{itemize}
Let $\L_n$ be the set of Laguerre histories of length $n$.
 There are several  well-known related such bijections 
between $\SS_n$ and 
$\L_n$, see \cite{Bi93, DV94, CSZ97} and references therein. We shall present  a variant of such encoding $\theta$, which is very close to Biane's version~\cite{Bi93}.    
 Our  strategy  is to first encode permutations using Laguerre 
 histories and then  build bijections $\rho_i$  on  the latter,
 where 
$\rho_i$   is the  bijection on $\L_n$ used in the proof of Theorem~$i$ 
($i=1, 2$). In otherwords, we have the following diagram
 $$
{\SS_n}\quad \xlongrightarrow{\theta }\quad{\L_n}\quad
\xlongrightarrow{\rho_i}\quad{\L_n}\quad
\xlongrightarrow{\theta^{-1}}\quad{\SS_n}.
$$

%%%%%

The paper is organized as follows: in Section~2 we present  permutation code, i.e., 
a bijection $\theta$ between $\SS_n$ and $\L_n$ with the main properties.
In Section~3 we construct the bijection $\rho_1$ from $\L_n$ onto itself
 and prove Theorem~1
by composing $\theta$ and $\rho_1$. In Section~4 we construct the bijection 
$\rho_2$ from $\L_n$ onto itself and prove Theorem~2
by composing $\theta$ and $\rho_2$.
%%%%%%%%%%%%
\section{Encoding permutations by Laguerre histories}       
%%%%%%%%%%%%
\begin{figure}[t]\label{fig1}
	\begin{center}
		\begin{tikzpicture}[scale=0.9]
		\draw[step=1cm, gray, thick, dotted] (0, 0) grid (17,4);
		
		\draw[black] (0, 0)--(1, 1)--(2, 2)--(3, 2)--(4, 2)--(5, 2)--(6, 3)--(7, 3)--(8, 3)--(9, 2)--(10, 1)--(11, 0)--(12, 0)--(13, 1)--(14, 2)--(15, 2)--(16,1)--(17,0);
		\draw[black] (0,0) node {$\bullet$};
		\draw[black] (1,1) node {$\bullet$};
		\draw[black] (2,2) node {$\bullet$};
		\draw[black] (3,2) node {$\bullet$};
		\draw[black] (4,2) node {$\bullet$};
		\draw[black] (5,2) node {$\bullet$};
		\draw[black] (6,3) node {$\bullet$};
		\draw[black] (7,3) node {$\bullet$};
		\draw[black] (8,3) node {$\bullet$};
		\draw[black] (9,2) node {$\bullet$};
		\draw[black] (10,1) node {$\bullet$};
		\draw[black] (11,0) node {$\bullet$};
		\draw[black] (12,0) node {$\bullet$};
		\draw[black] (13,1) node {$\bullet$};
		\draw[black] (14,2) node {$\bullet$};
		\draw[black] (15,2) node {$\bullet$};
		\draw[black] (16,1) node {$\bullet$};
		\draw[black] (17,0) node {$\bullet$};
		
		\draw [-,thin, blue](3,2)--(4,2);
		\draw [-,thin, black](2,2)--(3,2);
		\draw [-,thin, black](6,3)--(7,3);
		\draw [-,thin, black](7,3)--(8,3);
		\draw [-,thin, black](14,2)--(15,2);
		\draw [-,thin, red](4,2)--(5,2);
		\draw [-,thin, red](11,0)--(12,0);
		\tiny{
			\draw[black] (0.5,-0.5) node {$(\Delta,\Delta)$};
			\draw[black] (1.5,-0.5) node {$(\Delta,\Delta)$};
			\draw[black] (2.5,-0.5) node { $(\Delta,2)$};
			\draw[black] (3.5,-0.5) node {$(1,\Delta)$};
			\draw[black] (4.5,-0.5) node {$(\Delta,3)$};
			\draw[black] (5.5,-0.5) node {$(\Delta,\Delta)$};
			\draw[black] (6.5,-0.5) node {$(\Delta,1)$};
			\draw[black] (7.5,-0.5) node {$(\Delta,1)$};
			\draw[black] (8.5,-0.5) node {$(1,1)$};
			\draw[black] (9.5,-0.5) node {$(2,2)$};
			\draw[black] (10.5,-0.5) node {$(1,1)$};
			\draw[black] (11.5,-0.5) node {$(\Delta,1)$};
			\draw[black] (12.5,-0.5) node {$(\Delta,\Delta)$};
			\draw[black] (13.5,-0.5) node {$(\Delta,\Delta)$};
			\draw[black] (14.5,-0.5) node {$(\Delta,1)$};
			\draw[black] (15.5,-0.5) node {$(1,1)$};
			\draw[black] (16.5,-0.5) node {$(1,1)$};
			
			\draw[black] (2.5,2.5) node {$L_b$};
			\draw[black] (3.5,2.5) node {$L_a$};
			\draw[black] (4.5,2.5) node {$L_c$};
			\draw[black] (6.5,3.5) node {$L_b$};
			\draw[black] (7.5,3.5) node {$L_b$};
			\draw[black] (11.5,0.5) node {$L_c$};
			\draw[black] (14.5,2.5) node {$L_b$};
		}	
		\end{tikzpicture}
		\caption{The Laguerre history of  $\sigma=(1,4,11,7)(2,9,6,10,8,3)(5)(12)(13,16,14,17,15)$}
	\end{center}
\end{figure}
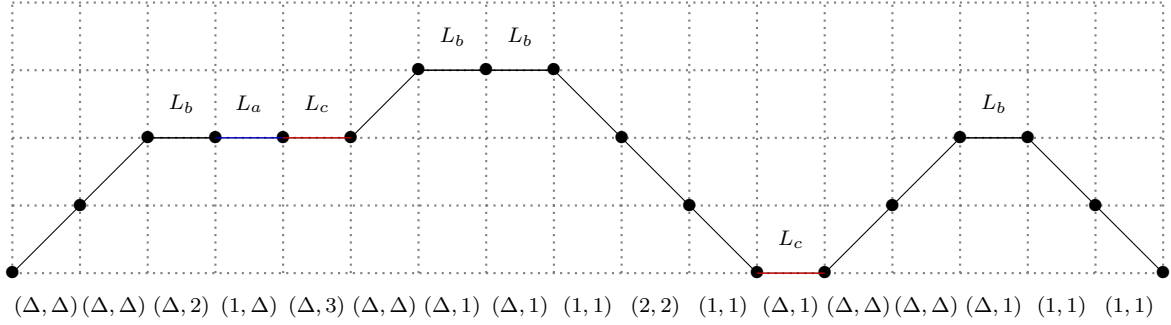 
%%%%%%%%%%%%%%%%%%%%%%

 For $\sigma \in \SS_n$ and $i\in [n]$,  the refined \emph{lower-nesting} and \emph{upper-nesting} numbers (see \cite{SZ19}) are defined by 
\begin{align*}
\lownest(i,\sigma)&=\#\{j\in [n]: \,\sigma(j)< \sigma(i)\; \textrm{and}\; i< j\},\\
\upnest(i,\sigma)&=\#\{j\in [n]:\,\sigma^{-1}(j)< \sigma^{-1}(i)\; \textrm{and}\;i < j\}.
\end{align*}
Note that $\upnest(i,\sigma)=\lownest(i,\sigma^{-1})$.

\begin{lem}
	Given $\sigma\in\SS_n$ and $i\in [n]$,
	we have 
	\begin{enumerate} 
		\item  $i\in \Arecp\sigma$ if and only if  $\lownest(i,\sigma)=0$,
		\item  $i\in \Recl\sigma$ if and only if  $\upnest(i,\sigma)=0$.
	\end{enumerate}
\end{lem}

\begin{proof}  Statement (1)  is clear by definition.
For $(i, \sigma) \in [n]\times \SS_n$, as  $ \upnest(i,\sigma)=\lownest(i,\sigma^{-1})$ 
we have 
 $$
 (i,  \sigma(i))\in \Arec\sigma\Leftrightarrow
		(\sigma(i), i)\in \Rec \sigma^{-1}.
				$$
		Hence  
		(2) is equivalent to (1).
   	\end{proof}
\noindent\textbf{Algorithm $\theta$}.
 For $\sigma\in\SS_n$, let $\theta(\sigma)=(s,\gamma)\in \L_n$ 
 where the pair $(s_i, \gamma_i)$ for $1\leq i\leq n$ is defined as follows:
 \begin{align}\label{def:laguerre-encoding}
(s_i, \gamma_i)=\begin{cases}
(\su, (\Delta, \Delta))&\textrm{if} \; i\in \Cval\sigma;
\\
(\sd, (\upnest(i,\sigma)+1, \lownest(i,\sigma)+1))&\textrm{if} \; i\in \Cpeak\sigma;
\\
(\la, (\upnest(i,\sigma)+1, \Delta))&\textrm{if} \; i\in \Cdrise\sigma;
\\
(\lb, (\Delta, \lownest(i,\sigma)+1))&\textrm{if} \; i\in \Cdfall\sigma;
\\
(\lc, (\Delta, h_{i-1}+1))&\textrm{if} \; i\in \Fix\sigma,
\end{cases}
\end{align}
%Set $\theta(\sigma)=(s,\gamma)$ with $s=(s_1,\ldots, s_n)$ and $\gamma=(\gamma_1, \ldots, \gamma_n)$. 
We can show that  $\theta$ is a variant of some well-known bijections from permutations to Laguerre hostories; see \cite{Bi93,DV94, CSZ97, SZ19} for  more details and other variants.  
\begin{lem}  
%The mapping $\theta: \SS_n\to \L_n$ is a bijection.  Moreover, 
Let  $\theta(\sigma)=(s,\gamma)$ and $h_0=0$. For $i=1,\ldots, n$,
\begin{itemize}
	\item if $i\in \Cval\sigma$,
	then $h_i=h_{i-1}+1$.
	\item if $i\in \Cpeak\sigma$,
	then $h_i=h_{i-1}-1$.
	\item if $i\in \Cdrise\sigma\cup\Cdfall\sigma\cup\Fix\sigma$,
	then $h_i=h_{i-1}$,
\end{itemize}
and $h_i=\#\{j\leq i: \sigma(j)>i\}=\#\{j\leq i: \sigma^{-1}(j)>i\}$.
\end{lem}

%There is a convenient way to compute the statistics $\lownest$ and $\upnest$ in 
%\eqref{def:laguerre-encoding}.
%For each integer $n\in \N$, let $[n']=\{1', \ldots, n'\}$ be a set of $n$ elements such that  $[n]\cap [n']=\emptyset$.    
A permutation 
  $\sigma\in\SS_n$  can be represented 
  by a \emph{bipartite digraph}  such that 
  \begin{itemize}
 \item the top row of vertices is  labelled by $1, \ldots, n$, 
 \item the bottom row of vertices is
 labelled by $1', \ldots, n'$, 
 \item there is an edge $i\to j'$ from the top row to the bottom row
 if and only if $\sigma(i)=j$.
 \end{itemize}
The bipartite digraph  associated to the running  example  is depicted  in Figure~\ref{graph-g}. We can visualize  the statistics $\lownest(i,\sigma)$ and $\upnest(i,\sigma)$ in \eqref{def:laguerre-encoding}  
by  the $i$-th restriction of the bipartite digraph  on 
${1\atop 1'}\cdots {i\atop i'}$
 for $i=1,\ldots, n$.  In other words, we have the following 
 graphical  description of
  the mapping $\theta$.\\
 
\medskip
\noindent
\textbf{Algorithm $\theta$ (bis).}  Let   $\sigma\in \SS_n$ and 
 $g_0=\emptyset$. For $i=1, \ldots, n$,  the $i$-th restriction $g_i$ is obtained from $g_{i-1}$ by 
  adding the column ${i\atop i'}$ at each time $i$ from left to right as follows:
	\begin{itemize}
		\item [(i)] 
		if $i\in \Cval\sigma$, 
		then $(s_i,\gamma_i)=(\su,(\Delta, \Delta))$,
		\item [(ii)] 
		if $i\in \Cpeak\sigma$, 
%		then there are two edges 
%		$\sigma^{-1}(i)\to i'$ and $i\to \sigma(i)'$, 
		then $(s_i,\gamma_i)=(\la, (\xi_i, \eta_i))$, where 
		$\xi_i-1$ is the number of vacant vertices to the left of 
		$\sigma^{-1}(i)$ at the top row of $g_i$ and 
		$\eta_i-1$ is the number of vacant vertices to the left of 
		$\sigma(i)$ at the bottom row of $g_i$.
		\item [(iii)] 
		if $i\in \Cdrise\sigma$, 
%		then there is an edge $\sigma^{-1}(i)\to i'$ 
%		 in $g_i$ with $i$ vacant,
		 then
		$(s_i,\gamma_i)=(\la, (\xi_i,\Delta))$, 
		where $\xi_i-1$ is the number of vacant vertices to the left of $\sigma^{-1}(i)$ at the top row of $g_{i}$,
		%$\xi_i=\#\{k<\sigma^{-1}(i)\,:\, \sigma(k)>i\}+1$, 
%		and $\upnest(i,\sigma)=k-1,$ 
%		connect $i'$ to $k$-th vacant vertex in $\{1, \ldots, i-1\}$.
%		Let $(s_i,\gamma_i)=(L_a, (k,\Delta)).$
		\item [(iv)] 
		if  $i\in \Cdfall\sigma$, 
		%then there is an edge $i\to \sigma(i)'$  in $g_i$, 
		then
		$(s_i,\gamma_i)=(\lb, (\Delta, \eta_i))$,
		where $\eta_i-1$ is the number of vacant vertices to  the left of 
		$\sigma(i)$ 
		 at the bottom row of $g_i$,
		%$\eta_i=\#\{k<\sigma^{-1}(i)\,:\, \sigma^{-1}(k)>i\}+1$,
%		and $\lownest(i,\sigma)=l-1,$
%		connect $i$ to $l$-th vacant vertex in $\{1', \ldots, i-1'\}$.
%		Let $(s_i,\gamma_i)=(L_a, (\Delta,l)).$
%		\item [(iv)] 
%		if $i\in \Fix\sigma$, 
%		%then $i\to i'$ is an edge in $g_i$, 
%		then $(s_i,\gamma_i)=(\lc, (\Delta, h_{i-1}+1))$,
		\item [(v)] 
		if  $i\in \Fix\sigma$, 
		then $(s_i,\gamma_i)=(\lc, (\Delta, h_{i-1}+1))$, where $h_{i-1}$ is the number of vacant vertices at the top (or bottom) row of $g_{i-1}$.
		%		$\xi_i=\#\{t<\sigma^{-1}(i)\,:\, \sigma(t)>i\}+1$ and 
%		$\eta_i=\#\{t<\sigma(i)\,:\, \sigma^{-1}(t)>i\}+1$.
			\end{itemize}

%%%%%%%%%%%%%%%%%%%%%%%%%
%%%%%%%%%%%%%%%%%%%%%%%%%%%%%%%%%%
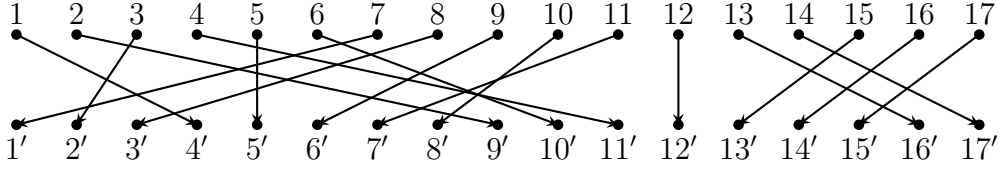
\begin{figure}[t]
	\begin{tikzpicture}[scale=0.8]
	\node[above] at (1,0.5) {$1$};
	\node[above] at (2,0.5) {$2$};
	\node[above] at (3,0.5) {$3$};
	\node[above] at (4,0.5) {$4$};
	\node[above] at (5,0.5) {$5$};
	\node[above] at (6,0.5) {$6$};
	\node[above] at (7,0.5) {$7$};
	\node[above] at (8,0.5) {$8$};
	\node[above] at (9,0.5) {$9$};
	\node[above] at (10,0.5) {$10$};
	\node[above] at (11,0.5) {$11$};
	\node[above] at (12,0.5) {$12$};
	\node[above] at (13,0.5) {$13$};
	\node[above] at (14,0.5) {$14$};
	\node[above] at (15,0.5) {$15$};
	\node[above] at (16,0.5) {$16$};
	\node[above] at (17,0.5) {$17$};
	%%%%%%%%%%%%%%%%%%%%%%%
	\fill (1,0.5) circle (2.5pt);
	\fill (2,0.5) circle (2.5pt);
	\fill (3,0.5) circle (2.5pt);
	\fill (4,0.5) circle (2.5pt);
	\fill (5,0.5) circle (2.5pt);
	\fill (6,0.5) circle (2.5pt);
	\fill (7,0.5) circle (2.5pt);
	\fill (8,0.5) circle (2.5pt);
	\fill (9,0.5) circle (2.5pt);
	\fill (10,0.5) circle (2.5pt);
	\fill (11,0.5) circle (2.5pt);
	\fill (12,0.5) circle (2.5pt);
	\fill (13,0.5) circle (2.5pt);
	\fill (14,0.5) circle (2.5pt);
	\fill (15,0.5) circle (2.5pt);
	\fill (16,0.5) circle (2.5pt);
	\fill (17,0.5) circle (2.5pt);
	%%%%%%%%%%%%%%%%%%%%%%%%%%%%%5
	\fill (1,-1) circle (2.5pt);
	\fill (2,-1) circle (2.5pt);
	\fill (3,-1) circle (2.5pt);
	\fill (4,-1) circle (2.5pt);
	\fill (5,-1) circle (2.5pt);
	\fill (6,-1) circle (2.5pt);
	\fill (7,-1) circle (2.5pt);
	\fill (8,-1) circle (2.5pt);
	\fill (9,-1) circle (2.5pt);
	\fill (10,-1) circle (2.5pt);
	\fill (11,-1) circle (2.5pt);
	\fill (12,-1) circle (2.5pt);
	\fill (13,-1) circle (2.5pt);
	\fill (14,-1) circle (2.5pt);
	\fill (15,-1) circle (2.5pt);
	\fill (16,-1) circle (2.5pt);
	\fill (17,-1) circle (2.5pt);
	%%%%%%%%%%%%%%%%%%%%%%%%%%%%%%%
	\node[below] at (1,-1) {$1'$};
	\node[below] at (2,-1) {$2'$};
	\node[below] at (3,-1) {$3'$};
	\node[below] at (4,-1) {$4'$};
	\node[below] at (5,-1) {$5'$};
	\node[below] at (6,-1) {$6'$};
	\node[below] at (7,-1) {$7'$};
	\node[below] at (8,-1) {$8'$};
	\node[below] at (9,-1) {$9'$};
	\node[below] at (10,-1) {$10'$};
	\node[below] at (11,-1) {$11'$};
	\node[below] at (12,-1) {$12'$};
	\node[below] at (13,-1) {$13'$};
	\node[below] at (14,-1) {$14'$};
	\node[below] at (15,-1) {$15'$};
	\node[below] at (16,-1) {$16'$};
	\node[below] at (17,-1) {$17'$};
	%%%%%%%%%%%%%%%%%%%%%%%55
	\draw [>=stealth,  thick,  ->](1,0.5)--(4,-1);
	\draw [>=stealth,  thick,  ->](2,0.5)--(9,-1);
	\draw [>=stealth,  thick,  ->](3,0.5)--(2,-1);
	\draw [>=stealth,  thick,  ->](4,0.5)--(11,-1);
	\draw [>=stealth,  thick,  ->](5,0.5)--(5,-1);
	\draw [>=stealth,  thick,  ->](6,0.5)--(10,-1);
	\draw [>=stealth,  thick,  ->](7,0.5)--(1,-1);
	\draw [>=stealth,  thick,  ->](8,0.5)--(3,-1);
	\draw [>=stealth,  thick,  ->](9,0.5)--(6,-1);
	\draw [>=stealth,  thick,  ->](10,0.5)--(8,-1);
	\draw [>=stealth,  thick,  ->](11,0.5)--(7,-1);
	\draw[>=stealth,  thick,  ->] (12,0.5)--(12,-1);
	\draw [>=stealth,  thick,  ->](13,0.5)--(16,-1);
	\draw [>=stealth,  thick,  ->](14,0.5)--(17,-1);
	\draw [>=stealth,  thick,  ->](15,0.5)--(13,-1);
	\draw [>=stealth,  thick,  ->](16,0.5)--(14,-1);
	\draw [>=stealth,  thick,  ->][->] (17,0.5)--(15,-1);
	\end{tikzpicture}
	\caption{ \small The  bipartite digraph of $\sigma=(1,4,11,7)(2,9,6,10,8,3)(5)(12)(13,16,14,17,15)$ }\label{graph-g}
\end{figure}

\begin{lem}\label{cycle pattern}  If $g$ is the bipartite digraph of $\sigma\in \SS_n$ and $g_i$ is the $i$-th restriction of $g$ ($1\leq i\leq n$), then
the index $i$ is a cycle maximum  of $\sigma$ if and only if 
either $i\to i'$ is an edge or
there are integers 
$i_1, i_2, \ldots, i_{k}$ in $[i]$ such that  
 $i\to i_{1}', i_{1}\to i_{2}', \ldots,i_{k}\to i'$ are edges of $g_i$.
\end{lem}
\begin{proof}
 Given  a permutation $\sigma\in \SS_n$, an index $i$ is a cycle maximum  if and only if
$\sigma^\ell(i)\leq i$  and 
for all $\ell\geq 0$, this means either $i\to i'$ is an edge or 
there are 
vertices 
$i_1, i_2, \ldots, i_{k}$ between $1$ and $i$
%in $\{1, \ldots, i\}$
 at the top row
and  $i_1', \ldots, i_{k}'$ at the bottom row such that  
 $i\to i_{1}', i_{1}\to i_{2}', \ldots,i_{k}\to i'$ are edges in the bipartite graph $g_i$.
\end{proof}

%The  bipartite digraphs and Laguerre histories
%are in one-to-one correspondence with permutations, 
\begin{rmk}
Starting from any permutation we can draw
the corresponding bipartite digraph or Laguerre history   through the 
above correspondences.  By the above lemma, the graphical interpretation 
enables us to to count  the cycle maxima of a permutation 
in  our bijections between Laguerre histories, which seems 
  difficult  in the corresponding Laguerre history via the bijection $\theta$.
\end{rmk}

 For each  Laguerre history $(s,\gamma)\in \L_n$ we define 
 the set-valued statistics:
\begin{align}
 \Arecp (s,\gamma)&=\{i: s_i\in \{\sd,  \lb, \lc\}\; \text{and}\; \eta_i=1\}\label{def:arecp}\\
		 \Erecl(s,\gamma)&=\{i: s_i=\sd\;\text{or}\;  \la, \text{and}\; \xi_i=1\}\\
		 \Erecp(s,\gamma)&=\{i: i\to j'\; \text{in g and}\; j\in \Erecl(s,\gamma)\}\\
		\Erec(s,\gamma)&=(\Erecp(s,\gamma), \Erecl(s,\gamma))\\
		 \Excp(s,\gamma)&=\{i: s_i=\su \;\text{or} \;\la\},\quad
		\Excl(s,\gamma)=\{i: s_i=\sd \;\text{or} \; \la\}\label{def:excpl}\\
		\Exc(s,\gamma)&=(\Excp(s,\gamma), \Excl(s,\gamma))\\
		 \Rar(s,\gamma)&=\{i: s_i=\lc \; \text{and}\;\eta_i=1 \}\\
		  \Cyc(s,\gamma)&=\{i: s_i=\lc \,\text{or}\, s_i=\sd \; \text{with}\; i\to i_{1}', i_{1}\to i_{2}', \ldots,i_{k}\to i'\; \text{in}\; g_i\}.\label{def:cyc}
	\end{align}

%	\begin{definition}\label{def:fondamental}  %\begin{itemize}
%\item $\Arecp (s,\gamma):=\{i: s_i\in \{\sd,  \lb, \lc\}\; \text{and}\; \eta_i=1\}$.
%		\item $\Erecl(s,\gamma):=\{i: s_i=\sd\;\text{or}\;  \la, \text{and}\; \xi_i=1\}$.
%		\item $\Rar(s,\gamma):=\{i: s_i=\lc \; \text{and}\;\eta_i=1 \}$.
%		\item $\Excp(s,\gamma):=\{i: s_i=\su \;\text{or} \;\la\}$.
%		\item 
%		$\Excl(s,\gamma):=\{i: s_i=\sd \;\text{or} \; \la\}$.
%		\item  $\Cyc(s,\gamma):=\{i: s_i=\lc \,\text{or}\; s_i=\sd \; \text{with}\; i\to i_{1}', i_{1}\to i_{2}', \ldots,i_{k}\to i'\; \text{in}\; g_i\}$.
%		\item
%		$\Erecp(s,\gamma):=\{i: i\to j'\; \text{in $g$ and}\; j\in \Erecl(s,\gamma)\}$.
%		\item 
%		$\Erec(s,\gamma):=(\Erecp(s,\gamma), \Erecl(s,\gamma))$.
%		%\{(i,j): i\in \Erecp\; \text{and}\; j\in \Erecl(s,\gamma)\}$.
%		\end{itemize}
%		\begin{itemize}
%		\item 
%		$\Exc(s,\gamma):=(\Excp(s,\gamma), \Excl(s,\gamma))$.
%		%\{(i,j): i\in \Excp\; \text{and}\; j\in \Excl(s,\gamma)\}$.
%		\end{itemize}
%\end{definition}

\begin{lem}\label{theta:properties}
	Let  $\sigma\in\SS_n$ and $\theta(\sigma)= (s,\gamma)$. 
	Then, % for $i\in [n]$,
	\begin{itemize}
		\item[(i)] $\Arecp (s,\gamma)
		%:=\{i: s_i\in \{\sd,  \lb, \lc\}\; \text{and}\; \eta_i=1\}
		=\Arecp\sigma$.
		\item[(ii)] $\Erecl(s,\gamma)
		%:=\{i: s_i=\sd\;\text{or}\;  \la, \text{and}\; \xi_i=1\}
		=\Erecl\sigma$.
		\item[(iii)] $\Rar(s,\gamma)%
		%:=\{i: s_i=\lc \; \text{and}\;\eta_i=1 \}
		=\Rar\sigma$.
		\item[(iv)] $\Excp(s,\gamma)
		=\Excp\sigma$.
		\item[(v)]
		$\Excl(s,\gamma)
		%:=\{i: s_i=\sd \;\text{or} \; \la\}
		=\Excl\sigma$.
		\item[(vi)] $\Cyc(s,\gamma)
		%:=\{i: s_i=\lc \,\text{or}\; s_i=\sd \; \text{with}\; i\to i_{1}', i_{1}\to i_{2}', \ldots,i_{k}\to i'\; \text{in}\; g_i\}
		=\Cyc\sigma$.
	\end{itemize}
\end{lem}

\begin{proof}
	Clearly, if 
	$i\in\Arecp\sigma$,
	then $\sigma(i)\leq i$, i.e., $i\in \Cdfall\sigma\cup\Cpeak\sigma\cup\Fix\sigma$,
	and if $i\in \Erecl\sigma$,
	then $i> \sigma^{-1}(i)$, i.e., $i\in \Cdrise\sigma\cup \Cpeak\sigma$.
	By Lemma~3 and (2.1) we get (i) and (ii).
	Next, an integer $i\in \Rar\sigma
	\Leftrightarrow
	i\in \Fix\sigma\cap\Arecp\sigma$,
	so by (i) and \eqref{def:laguerre-encoding} we get (iii).
	Since $\Excp=\Cval\cup \Cdrise$ and $\Excl=\Cpeak\cup \Cdrise$,
	by \eqref{def:laguerre-encoding}, we get (iv) and (v).
	By Lemma~5 we get (vi).
	\end{proof}

%\newpage

\section{Proof of Theorem 1}
\subsection{Algorithm $\rho_1$}\label{algo-rho1}
For $(s,\gamma)\in \L_n$, 
we define $\rho_1(s,\gamma)=(s',\gamma')$  through 
the corresponding bipartite digraphs $g'_i$ for $1\leq i\leq n$.
Set $g'_0=\emptyset$.
The graph $g'_{i}$ is
obtained   from   $g'_{i-1}$ by first adding the column 
 ${i\atop i'}$ with  possibles edges as follows:

\begin{itemize}
	\item [(i)-(ii)] 
	If $s_i=\su$ or $\la$, then $(s_i', \gamma')=(s_i, \gamma_i)$.
	%		and $\gamma_i=\Delta$, then 
	%		$(s'_i, \gamma'_i)=(s_i, \gamma_i)$.
	%\item [(ii)] 
	%		If  $s_i=\la$ 
	%		and $\gamma_i=(\xi_i,\Delta)$, then   $(s'_i, \gamma'_i)=(\la, (\xi_i,\Delta))$.
	\item [(iii)] 
	If  $s_i=\lb$ and $\gamma_i=(\Delta,\eta_i)$,
	then 
	\begin{align}
	(s_i', (\xi'_i, \eta_i'))=\begin{cases}
	(\lc, (\Delta, h_{i-1}+1))& \text{if $\eta_i=1$}\\
	(\lb, (\Delta, \eta_i))& \text{if $\eta_i>1$}.
	\end{cases}
	\end{align}
	Note that $\eta_i'>1$ because  $h_{i-1}\geq 1$.
	%     \begin{enumerate}
	%		\item[(a)] $\eta_i=1$,
	%		then  $(s'_i, \gamma'_i)=(\lc,  (\Delta, 1))$.
	%		\item[(b)] $\eta_i>1$, then    $(s'_i, \gamma'_i)=(\lb, (\Delta, \eta_i))$.
	%       \end{enumerate}
	\item [(iv)] 
	If  $s_i=\lc$ and $\gamma_i=(\Delta, h_{i-1}+1)$,
	then 
	\begin{align}
	(s_i', (\xi'_i, \eta_i'))=\begin{cases}
	(\lc, (\Delta,1))& \text{if $\eta_i=1$}\\
	(\lb, (\Delta, 1))& \text{if $\eta_i>1$}.
	\end{cases}
	\end{align}
 \item [(v)] 
		If  $s_i=\sd$ and $\gamma_i=(\xi_i,\eta_i)$, let $s_i'=s_i$ and 
		$\xi_i'=\xi_i$.  Thus 
		the vertex $i'$ is connected to the $\xi_i$-th vacant vertex  $v_1$ at the top row of $g_i'$ and there are edges  $i'\to v_1, v'_1\to v_2, \ldots,
		v'_{r-1}\to v_r$ such that  $v_r'$ is vacant. Assume that  $v_r'$ is  the $\eta^*_i$-th 
		vacant vertex  at the bottom row  from left to right.  Then
		\begin{align}\label{V}
		\eta_i'=\begin{cases}
		\eta_i^*&\text{if $\eta_i=1$}\\
		\eta_i&\text{if $\eta_i>\eta_i^*$}\\
		\eta_i-1&\text{if $1<\eta_i\leq \eta_i^*$}.
		\end{cases}
		\end{align}
%		\begin{enumerate}
%		\item[(a)] if $\eta_i=1$,
%		 let $\eta'_i=\eta^*_i$;
%		\item[(b)] if $\eta_i>\eta^*_i$, let $\eta'_i=\eta_i$ 
%		%(i.e., connect $i$ to $\eta_i$-th vacant vertex in $\{1', \ldots, i'\}$);
%		\item[(c)] if $1<\eta_i\leq \eta^*_i$,
%		let  $\eta'_i=\eta_i-1$.
%		%(i.e., connect $i$ to the $\eta'_i$-th vacant vertex 
%		%in $\{1', \ldots, i'\}$).
%		\end{enumerate}
	\end{itemize}
Set $\rho_1(s, \gamma)=(s', \gamma')$, where $s'=(s_1', \ldots, s_n')$ and
$\gamma'=(\gamma_1', \ldots, \gamma_n')$.

\begin{lem}
	The mapping $\rho_1: \L_n\to \L_n$ is a bijection.
	\end{lem}
\begin{proof}  
We construct the inverse  mapping $\rho_1^{-1}$. Let 
$(s', \gamma')=\rho_1(s,\gamma)$ and $g''_0=\emptyset$.
	For $1\leq i\leq n$, 
	if $s'_i\in \{\su,\la,\lb,\lc\}$,
	we define $(s''_i, \gamma''_i)$ 
	in the same way as $\rho_1$; 
	if $s'_i=\sd$ 
    and $\gamma'_i=(\xi'_i, \eta'_i)$, let $s_i''=D$ and $\xi_i''=\xi'_i$.
	Let $v_r$ be the unique vertex such that
    $i'\gets v_1, v'_1\gets v_2, \ldots,v'_{r-1}\gets v_r$ and $v'_r\gets i$ or  $v'_r$ is vacant   in $g'_{i}$. 
	Assume that there are  $\tilde{\eta}^*_i-1$ vacant vertices at the left of $v'_r$
at the bottom row 
	of  $g'_{i}$.  Then
			\begin{align}\label{V-1}
		\eta_i''=\begin{cases}
		1&\text{if $\eta_i'=\tilde{\eta}_i^*$}\\
		\eta_i'&\text{if $\eta_i'>\tilde{\eta}_i^*$}\\
		\eta_i'+1&\text{if $1\leq \eta_i'<\tilde{\eta}_i^*$}.
		\end{cases}
		\end{align}
		Set $\rho_1^{-1}(s', \gamma')=(s'', \gamma'')$, where $s''=(s_1'', \ldots, s_n'')$ and
$\gamma''=(\gamma_1'', \ldots, \gamma_n'')$.
% For $i=1, \ldots, n$, let  $\rho_1(s_i,\gamma_i)=(s_i',\gamma_i')$ and
%  $\rho_1^{-1}(s'_i,\gamma'_i)=(\bar{s}_i,\bar{\gamma}_i)$.
  We show that  $(s''_i, \gamma''_i)=(s_i,\gamma_i)$
 by induction on $i$.
 If $i=1$, then  $(s_1,\gamma_1)=(\su, (\Delta, \Delta))$
 or $(\lc,(1,1))$. By definition of $\rho_1$ (see (i) and (iv) (a)) and $\rho^{-1}_1$, 
  for both two cases,
 we get $(\bar{s}_1,\bar{\gamma}_1)=(s_1,\gamma_1)$.
 For $i\geq 2$ assume 
 $(s''_k,\gamma''_k)=(s_k,\gamma_k)$ for $1\leq k\leq i-1$.
We have to verify  the validity for $k=i$ for the five types of $s_i$ in \S~\ref{algo-rho1}. This  is  trivial for 
the cases (i)-(iv).  For  case (v), we have 
 	$s_i=s_i'=s''_i=D$ and $\xi''_i=\xi_i'=\xi_i$. 
 	By (v), the mapping $\rho_1$ provides
 	the index  $\eta^*_i$, which is equal to $\tilde{\eta}^*_i$.
	There are three cases~:
 	\begin{itemize}
 		\item %[(a)-$(\bar a)$.] 
		$\eta_i=1$.
 	    Then $\eta'_i=\eta^*_i$ by \eqref {V}, this means 
 	    $\eta''_i=1$  by  \eqref{V-1}.
 		\item %[(b)-$(\bar b)$.]
		 $\eta_i>\eta^*_i$.
	    Then $\eta'_i=\eta_i$,
	    and $\bar{\eta}_i=\eta'_i=\eta_i$ by  \eqref {V}and 
	    \eqref{V-1}.
 		\item %[(c)-$(\bar c)$.] 
		$1<\eta_i\leq \eta^*_i$.
 		Then $\eta'_i=\eta_i-1$ by  \eqref {V}
 		and $\bar{\eta}_i=\eta'_i+1=\eta_i$  by 
	   \eqref{V-1}.
 	\end{itemize}
Summarizing the above five cases, we have completed the proof.
%we prove $(\bar{s}_i, \bar{\gamma}_i)=(s_i, \gamma_i)$ for all $1\leq i\leq n$.
%So we prove that $\rho_1'=\rho_1^{-1}$ and $\rho_1$ is a bijection from $\L_n$ to $\L_n$.
 \end{proof}
 
 \begin{rmk}
 	  Using the graph $g_i'$ we determine  $\eta^*_i$ by  (3.3) and 	
	  characterize  $i\in\Cyc(s',\gamma')$ by the equation $\eta'_i=\eta^*_i$. 
	  Next, 
 	 using $g_i'$ we get the index $\tilde{\eta}_i^*$ by (3.4)  and identify $i\in\Cyc(s',\gamma')$
 	by the equation $\eta'_i=\tilde{\eta}_i^*$.
 	Thus  $\eta^*_i=\tilde{\eta}_i^*$.
 	\end{rmk}

\begin{lem}\label{rho1:properties}
For $(s,\gamma)\in \L_n$, if   $\rho_1(s,\gamma)=(s',\gamma')$, then
\begin{equation*}
	(\Cyc ,\Exc ,\Erecl ,\Rar)(s',\gamma')
	=(\Arecp,\Exc,\Erecl,\Rar)
	(s,\gamma).  
	\end{equation*}	      
\end{lem}
\begin{proof} 
By definition of $\rho_1$ in \S~\ref{algo-rho1}, if 
	$\rho_1(s_i,\gamma_i)=(s_i',\gamma_i')$ for $1\leq i\leq n$, then 
	$s_i=s_i'$ for  $s_i\in \{\su, \sd, \la\}$ and 
		$
		\{i: s_i=\lb\,\;\textrm{or}\,\;\lc\}=\{i: s'_i=\lb\,\;\textrm{or}\,\;\lc\}
		$.
	\begin{enumerate}
		\item 
		As $\Excp(s,\gamma)=\{i: s_i=\su\;\text{or}\; \la\}$ and 
		$\Excl(s,\gamma)=\{i: s_i=\sd\;\text{or}\; \la\}$, see  \eqref{def:arecp}-\eqref{def:cyc}, we have 
		$\Exc (s',\gamma')=\Exc (s,\gamma)$.
%		By (i)-(v), we have 
		\item  Recall that 
		$\Arecp (s,\gamma):=\{i: s_i\in \{\sd,  \lb, \lc\}\; \text{and}\; \eta_i=1\}$
		and  $\Cyc(s',\gamma')=\{i: s'_i=\lc \,\text{or}\, s'_i=\sd \; \text{with}\; i\to i_{1}', i_{1}\to i_{2}', \ldots,i_{k}\to i'\; \text{in}\; g'_i\}$ (cf. \eqref{def:arecp} and \eqref{def:cyc}).
		%.
		\begin{itemize}
			\item 
	        $(s_i,\gamma_i)=(\lb,(\Delta,1))$ or $(\lc,(\Delta,1))$ 
			if and only if $(s'_i,\gamma'_i)=(\lc,(\Delta,h_{i-1}+1))$, in other word there is an edge $i\to i'$ in $g_i'$ by (iii) and (iv).
			\item  $(s_i,\gamma_i)=(\sd,(\xi_i,1))$
			if and only if
			$s'_i=\sd$,
			$\xi'_i=\xi_i$,
			and in $g_i'$
			there are edges $i\to i_{1}', i_{1}\to i_{2}', \ldots,i_{k}\to i'$ by \S~\ref{algo-rho1} (v).
		\end{itemize}
		Thus $\Cyc (s',\gamma')=\Arecp (s,\gamma)$.
		\item Recall that 
		 $\Erecl(s,\gamma):=\{i: (s_i, \xi_i)=(\sd, 1)\,\text{or}\,  (\la, 1)\}$.
		By \S~\ref{algo-rho1} (ii)(v) the latter is equivalent to
		%$(s_i,\gamma_i)=(\la,(1,\Delta))$ or $(\sd,(1,\eta_i))$ 
		%if and only if 
		$(s'_i,\xi'_i)=(\sd, 1)\,\text{or}\,  (\la, 1)\}$.
		So $\Erecl (s',\gamma')=\Erecl (s,\gamma)$.
		\item 
		By \S~\ref{algo-rho1} (iv),
		$(s_i,\gamma_i)=(\lc,(\Delta,1))$
		if and only if $(s'_i,\gamma'_i)=(\lc,(\Delta,1))$.  Thus 
		$\Rar (s',\gamma')=\Rar (s,\gamma)$.
	\end{enumerate}
\end{proof}

\begin{lem}\label{rho1:moreproperties}
	For $(s,\gamma)\in \L_n$, if   $\rho_1(s,\gamma)=(s',\gamma')$, then
	$\Erecp(s',\gamma')
	=\Erecp(s,\gamma)$.        
\end{lem}
	\begin{proof} 
	Let $g$ and $g'$ denote  the bipartite digraphs
	 corresponding to $(s,\gamma)$ and $(s',\gamma')$. 
	 The construction  of $\rho_1$ shows that   $\xi_i'=\xi_i$
		for $1\leq i\leq n$.
	 Let  (cf. Lemma \ref{rho1:properties})
		\begin{align}
		 P:&=\Excp(s',\gamma')=\Excp(s,\gamma)=\{u_1,\ldots,u_t\}_{<} \label{P1}\\
		L:&=\Excl(s',\gamma')=\Excl(s,\gamma)=\{v_1,\ldots,v_t\}_{<}.\label{L1}
		\end{align}
		Since  $\Erecl(s,\gamma)\subset L$ (see (2.3) and  (2.6)) and 
        $\Erecl(s',\gamma')
		=\Erecl(s,\gamma)$ (see  Lemma~8),
%		thus we have $\Erecp(s',\gamma')
%		=\Erecp(s,\gamma)$.
		 it suffices to prove  that 
		$u_j\to v_i'$  is an edge in  $g$  with 
		$(u_j, v_i)\in  P\times L$ if and only if $u_j\to v_i'$  is an edge in $g'$.
		Let $E_i$  (resp. $E'_i$) be the set of vacant vertices at the top row of 
		$g_{v_i-1}$ (resp. $g'_{v_i-1}$).
		We show that $E_i=E'_i$ for all $v_i\in L$
		for $i=1,\ldots,t$.
		\begin{itemize}
		\item For $i=1$,  a vertex $v$  at the top row of $g_{v_1-1}$ 
		(resp. $g'_{v_1-1}$)  is 
		vacant if and only if $v\in P$ and $v<v_1$. Thus 
		$E_1=E_1'$ by \eqref{P1}.  
		\item 
		Assume $E_{i-1}=E_{i-1}'$.
		 If $u_j\to v_{i-1}'$ is an edge  in  $g_{v_{i-1}}$, then it is also in $g'$ because 
		  $\xi_{v_{i-1}}=\xi'_{v_{i-1}}$. Thus
		  $
		  E_i=\left(E_{i-1}\setminus \{u_j\}\right)\cup \{v\in P: v_{i-1}\leq v\leq v_i-1\}.
		  $
		So $E_{i}=E_{i}'$.
			\end{itemize}
					
					Since $\xi_{v_i}=\xi'_{v_i}$ for $v_i\in L$,
					each $v_i'$ is connected to the same  vertex 
		in both $g$ and $g'$. 
	\end{proof}
	
	By Lemma~8 and 9,
	for $(s,\gamma)\in \L_n$, if   $\rho_1(s,\gamma)=(s',\gamma')$, then
	\begin{equation*}
	(\Cyc ,\Exc ,\Erec ,\Rar)(s',\gamma')
	=(\Arecp,\Exc,\Erec,\Rar)
	(s,\gamma).  
	\end{equation*}	      
%We  consider the four statistics separately.
%\begin{enumerate}
%\item $\Exc\varphi(\sigma)=\Exc\sigma.$
%We have proved that $\rho_1:\L_n\to \L_n$ is a %bijection  such that
%$s_i=s_i'$ if $s_i=\{\su, \sd, \la\}$ and 
%$
%\{i: s_i=\lb\,\;\textrm{or}\,\;\lc\}=\{i: %s'_i=\lb\,\;\textrm{or}\,\;\lc\}.
%$
%Moreover,  the height sequences $(h_i)$ and %$(h_i')$ are identical.
%Thus by Lemma~6 (iv),
%we have $\Exc\varphi(\sigma)=\Exc\sigma$.
%\item $\Cyc\varphi(\sigma)=\Arecp\sigma$.
%If $ i\in \Arecp\sigma$, 
%then  $i\in \Cdfall\sigma\cup \Cpeak\sigma %\cup \Fix\sigma$, and in view of 
%(iii)-(v) and Lemma~6 (i),(v), 
%\begin{itemize}	
%\item if $i\in \Arecp\sigma$, then $i\in %\Cyc\varphi(\sigma)$,
	%\item if $i\notin \Arecp\sigma$, then %$i\notin \Cyc\varphi(\sigma)$.
%\end{itemize}
%So we have $\Cyc\varphi(\sigma)=\Arecp\sigma$. 
%%%
%\item $\Erec\varphi(\sigma)=\Erec\sigma$.
%Since $\xi_i'=\xi_i$ for all $i\in [n]$, by %Lemma~6 (ii),
%we have $\Erecl\varphi(\sigma)=\Erecl\sigma.$
%It is easy to see that if $i\in \Erecp\sigma$, %then $i\in \Excp\sigma$.
%So by Lemma~7,
%we have $\Erecp\pi=\Erecp\sigma,$
%thus we get
%$\Erec\pi=\Erec\sigma.$
%\item $\Rar\varphi(\sigma)=\Rar\sigma.$ 
%If $i\in \Rar \sigma$, then $s_i=\lc$ and %$\xi_i=\eta_i=1$.
%By Case $(iv)$
%we have  $s'_i=\lc$ and $\xi_i'=\eta'_i=1$,
%thus  $\Rar\varphi(\sigma)=\Rar\sigma$.
%\end{enumerate}

%%%%%%%%%%%%%%%%%%%%%%%%%
\begin{figure}[t]
	\begin{tikzpicture}[scale=0.8]
	\node[above] at (1,0.5) {$1$};
	\node[above] at (2,0.5) {$2$};
	\node[above] at (3,0.5) {$3$};
	\node[above] at (4,0.5) {$4$};
	\node[above] at (5,0.5) {$5$};
	\node[above] at (6,0.5) {$6$};
	\node[above] at (7,0.5) {$7$};
	\node[above] at (8,0.5) {$8$};
	\node[above] at (9,0.5) {$9$};
	\node[above] at (10,0.5) {$10$};
	\node[above] at (11,0.5) {$11$};
	\node[above] at (12,0.5) {$12$};
	\node[above] at (13,0.5) {$13$};
	\node[above] at (14,0.5) {$14$};
	\node[above] at (15,0.5) {$15$};
	\node[above] at (16,0.5) {$16$};
	\node[above] at (17,0.5) {$17$};
	%%%%%%%%%%%%%%%%%%%%%%%
	\fill (1,0.5) circle (2.5pt);
	\fill (2,0.5) circle (2.5pt);
	\fill (3,0.5) circle (2.5pt);
	\fill (4,0.5) circle (2.5pt);
	\fill (5,0.5) circle (2.5pt);
	\fill (6,0.5) circle (2.5pt);
	\fill (7,0.5) circle (2.5pt);
	\fill (8,0.5) circle (2.5pt);
	\fill (9,0.5) circle (2.5pt);
	\fill (10,0.5) circle (2.5pt);
	\fill (11,0.5) circle (2.5pt);
	\fill (12,0.5) circle (2.5pt);
	\fill (13,0.5) circle (2.5pt);
	\fill (14,0.5) circle (2.5pt);
	\fill (15,0.5) circle (2.5pt);
	\fill (16,0.5) circle (2.5pt);
	\fill (17,0.5) circle (2.5pt);
	%%%%%%%%%%%%%%%%%%%%%%%%%%%%%5
	\fill (1,-1) circle (2.5pt);
	\fill (2,-1) circle (2.5pt);
	\fill (3,-1) circle (2.5pt);
	\fill (4,-1) circle (2.5pt);
	\fill (5,-1) circle (2.5pt);
	\fill (6,-1) circle (2.5pt);
	\fill (7,-1) circle (2.5pt);
	\fill (8,-1) circle (2.5pt);
	\fill (9,-1) circle (2.5pt);
	\fill (10,-1) circle (2.5pt);
	\fill (11,-1) circle (2.5pt);
	\fill (12,-1) circle (2.5pt);
	\fill (13,-1) circle (2.5pt);
	\fill (14,-1) circle (2.5pt);
	\fill (15,-1) circle (2.5pt);
	\fill (16,-1) circle (2.5pt);
	\fill (17,-1) circle (2.5pt);
	%%%%%%%%%%%%%%%%%%%%%%%%%%%%%%%
	\node[below] at (1,-1) {$1'$};
	\node[below] at (2,-1) {$2'$};
	\node[below] at (3,-1) {$3'$};
	\node[below] at (4,-1) {$4'$};
	\node[below] at (5,-1) {$5'$};
	\node[below] at (6,-1) {$6'$};
	\node[below] at (7,-1) {$7'$};
	\node[below] at (8,-1) {$8'$};
	\node[below] at (9,-1) {$9'$};
	\node[below] at (10,-1) {$10'$};
	\node[below] at (11,-1) {$11'$};
	\node[below] at (12,-1) {$12'$};
	\node[below] at (13,-1) {$13'$};
	\node[below] at (14,-1) {$14'$};
	\node[below] at (15,-1) {$15'$};
	\node[below] at (16,-1) {$16'$};
	\node[below] at (17,-1) {$17'$};
	%%%%%%%%%%%%%%%%%%%%%%%55
	\draw [>=stealth,  thick,  ->](1,0.5)--(4,-1);
	\draw [>=stealth,  thick,  ->](2,0.5)--(9,-1);
	\draw [>=stealth,  thick,  ->](3,0.5)--(2,-1);
	\draw [>=stealth,  thick,  ->](4,0.5)--(11,-1);
	\draw [>=stealth,  thick,  ->](5,0.5)--(1,-1);
	\draw [>=stealth,  thick,  ->](6,0.5)--(10,-1);
	\draw [>=stealth,  thick,  ->](7,0.5)--(7,-1);
	\draw [>=stealth,  thick,  ->](8,0.5)--(8,-1);
	\draw [>=stealth,  thick,  ->](9,0.5)--(3,-1);
	\draw [>=stealth,  thick,  ->](10,0.5)--(5,-1);
	\draw [>=stealth,  thick,  ->](11,0.5)--(6,-1);
	\draw [>=stealth,  thick,  ->] (12,0.5)--(12,-1);
	\draw [>=stealth,  thick,  ->](13,0.5)--(16,-1);
	\draw [>=stealth,  thick,  ->](14,0.5)--(17,-1);
	\draw [>=stealth,  thick,  ->](15,0.5)--(15,-1);
	\draw [>=stealth,  thick,  ->](16,0.5)--(13,-1);
	\draw [>=stealth,  thick,  ->][->] (17,0.5)--(14,-1);
	\end{tikzpicture}
	\caption{ \small The  graph of $\varphi(\sigma)=(1,4,11,6,10,5)(7)(8)(2,9,3)(12)(13,16)(14,17)(15)$ }\label{graph-g1}
\end{figure}
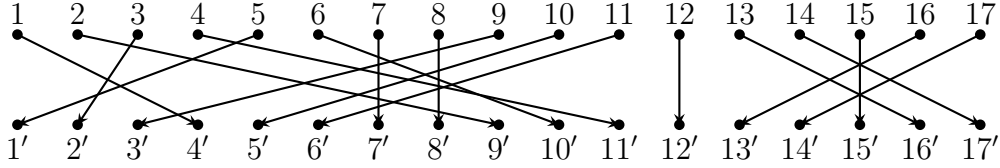

%%%%%2%%%%
\begin{figure}[t]
	\begin{center}
		\begin{tikzpicture}[scale=0.9]
		\draw[step=1cm, gray, thick,dotted] (0, 0) grid (17,4);
		
		\draw[black] (0, 0)--(1, 1)--(2, 2)--(3, 2)--(4, 2)--(5, 2)--(6, 3)--(7, 3)--(8, 3)--(9, 2)--(10, 1)--(11, 0)--(12, 0)--(13, 1)--(14, 2)--(15, 2)--(16,1)--(17,0);
		\draw[black] (0,0) node {$\bullet$};
		\draw[black] (1,1) node {$\bullet$};
		\draw[black] (2,2) node {$\bullet$};
		\draw[black] (3,2) node {$\bullet$};
		\draw[black] (4,2) node {$\bullet$};
		\draw[black] (5,2) node {$\bullet$};
		\draw[black] (6,3) node {$\bullet$};
		\draw[black] (7,3) node {$\bullet$};
		\draw[black] (8,3) node {$\bullet$};
		\draw[black] (9,2) node {$\bullet$};
		\draw[black] (10,1) node {$\bullet$};
		\draw[black] (11,0) node {$\bullet$};
		\draw[black] (12,0) node {$\bullet$};
		\draw[black] (13,1) node {$\bullet$};
		\draw[black] (14,2) node {$\bullet$};
		\draw[black] (15,2) node {$\bullet$};
		\draw[black] (16,1) node {$\bullet$};
		\draw[black] (17,0) node {$\bullet$};
		
		\draw [-,thin, blue](3,2)--(4,2);
		\draw [-,thin, black](2,2)--(3,2);
		\draw [-,thin, red](6,3)--(7,3);
		\draw [-,thin, red](7,3)--(8,3);
		\draw [-,thin, red](14,2)--(15,2);
		\draw [-,thin, black](4,2)--(5,2);
		\draw [-,thin, red](11,0)--(12,0);
		
		\tiny{
			\draw[black] (0.5,-0.5) node {$(\Delta,\Delta)$};
			\draw[black] (1.5,-0.5) node {$(\Delta,\Delta)$};
			\draw[black] (2.5,-0.5) node {$(\Delta,2)$};
			\draw[black] (3.5,-0.5) node {$(1,\Delta)$};
			\draw[black] (4.5,-0.5) node {$(\Delta,1)$};
			\draw[black] (5.5,-0.5) node {$(\Delta,\Delta)$};
			\draw[black] (6.5,-0.5) node {$(\Delta,4)$};
			\draw[black] (7.5,-0.5) node {$(\Delta,4)$};
			\draw[black] (8.5,-0.5) node {$(1,1)$};
			\draw[black] (9.5,-0.5) node {$(2,1)$};
			\draw[black] (10.5,-0.5) node {$(1,1)$};
			\draw[black] (11.5,-0.5) node {$(\Delta,1)$};
			\draw[black] (12.5,-0.5) node {$(\Delta,\Delta)$};
			\draw[black] (13.5,-0.5) node {$(\Delta,\Delta)$};
			\draw[black] (14.5,-0.5) node {$(\Delta,3)$};
			\draw[black] (15.5,-0.5) node {$(1,1)$};
			\draw[black] (16.5,-0.5) node {$(1,1)$};
			
			\draw[black] (2.5,2.5) node {$L_b$};
			\draw[black] (3.5,2.5) node {$L_a$};
			\draw[black] (4.5,2.5) node {$L_b$};
			\draw[black] (6.5,3.5) node {$L_c$};
			\draw[black] (7.5,3.5) node {$L_c$};
			\draw[black] (11.5,0.5) node {$L_c$};
			\draw[black] (14.5,2.5) node {$L_c$};
		}	
		\end{tikzpicture}
		\caption{ \small The Laguerre history of  $\varphi(\sigma)$}
		%=(1,4,11,6,10,5)(2,9,3)(7)(8)(12)(13,16)(14,17)(15)$}
	\end{center}
\end{figure}
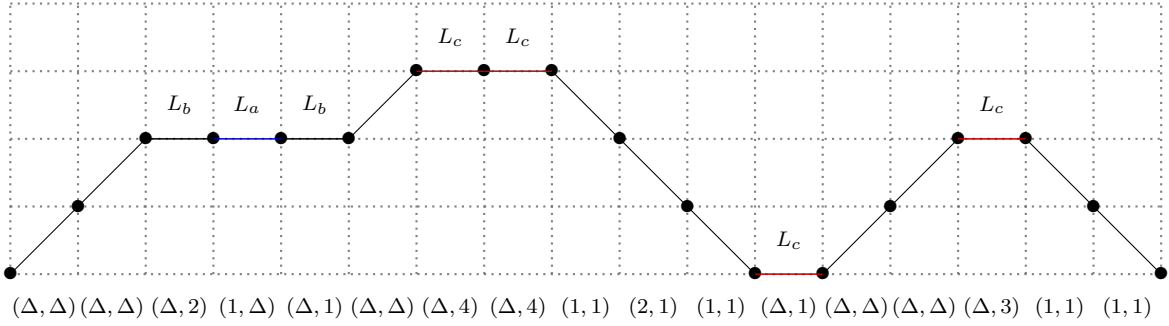 

Let $\varphi:=\theta^{-1}\circ\rho_1\circ\theta$. 
By Lemmas~\ref{theta:properties},  \ref{rho1:properties} and 9,   we see that 
$\varphi$ is the desired bijection from $\SS_n$ onto itself for Theorem~1.
\begin{ex}
For our running example $\sigma$,
we get  $\omega:=\varphi(\sigma)$ with $$\omega=\left(%
\begin{array}{cccccccccccccccccc}
1 & 2 & 3 & 4 & 5 & 6 & 7 & 8 & 9 & 10 & 11 & 12 & 13 & 14 & 15 & 16 & 17 \\
4 & 9 & 2 & 11 & 1 & 10 & 7 & 8 & 3 & 5 & 6 & 12 & 16 & 17 & 15 & 13 & 14  \\
\end{array}%
\right).$$
We have 
\begin{align*}
\Rec\omega&=(\{
1, \,2, \,4, \,12, \,13, \,14\},\,
\{4, \,9, \,11, \,12, \,16, \,17
\})\\
\Arec\omega&=(\{
5, \,9, \,10, \,11, \,12, \,16, \,17\},\,
\{1, \,3, \,5, \,6, \,12, \,13, \,14
\})\\
\Erec\omega&=(\{1, \,2, \,4, \,13, \,14\},\, \{4, \,9, \,11, \,16, \,17\})\\
\Exc\omega&=(\{
1,  \,2, \,4, \,6, \,13, \,14\},\, 
\{4, \,9, \,11, \,10, \,16, \,17\})
\end{align*}
and $\Rar\omega=\{12\}$.
Also, the cycle decomposition of $\omega$   is 
\begin{equation*}
\omega=(1,4,11,6, 10, 5)\,(7)\,(8)\,(2,9,3)\,(12)\,(13,16)\,(14,17)\,(15),
\end{equation*}
so,
$\Cyc\omega=\{7,8,9,11,12,15,16, 17\}$,  
$\Cpeak\omega=\{9,10,11,16,17\}$,
$\Cval\omega=\{1,2,6,13,14\}$,
$\Cdrise\omega=\{4\}$,
$\Cdfall\omega=\{3,5\}$,  and  
$\Fix\omega=\{7,8,12,15\}$. 
%The associated graphs  are depicted  in Figure~\ref{fig5}.

\end{ex}

%%%%%%%%%%%%%%%%%%%%%%%%%%
\section{Proof of Theorem 2}
%%%%%%%%%%%%%%%%%%%%%%%%
\subsection{Algorithm $\rho_2$}\label{def:rho2}
For $(s, \gamma)\in \L_n$, we define $\rho_2(s, \gamma)=(s',\gamma')$
through 
the corresponding bipartite digraphs $(g_i, g'_i)$ for $1\leq i\leq n$ as follows.
%Set $g_0'=\emptyset$, $\ell'_0=\emptyset$.
%the graph $g'_{i}$ is
%obtained   from   $g'_{i-1}$ as follows:
\begin{itemize}
\item [(i)-(ii)] 
		If $s_i=\su$ or $\la$, then $(s_i', \gamma')=(s_i, \gamma_i)$.
%		and $\gamma_i=\Delta$, then 
%		$(s'_i, \gamma'_i)=(s_i, \gamma_i)$.
%\item [(ii)] 
%		If  $s_i=\la$ 
%		and $\gamma_i=(\xi_i,\Delta)$, then   $(s'_i, \gamma'_i)=(\la, (\xi_i,\Delta))$.
\item [(iii)] 
		If  $s_i=\lb$ and $\gamma_i=(\Delta,\eta_i)$,
		 then 
		 \begin{align}
		 (s_i', (\xi'_i, \eta_i'))=\begin{cases}
		 (\lc, (\Delta, h_{i-1}+1))& \text{if $\eta_i=1$}\\
		 (\lb, (\Delta, \eta_i))& \text{if $\eta_i>1$}.
		 \end{cases}
		 \end{align}
		 Note that $\eta_i'>1$ because  $h_{i-1}\geq 1$.
%     \begin{enumerate}
%		\item[(a)] $\eta_i=1$,
%		then  $(s'_i, \gamma'_i)=(\lc,  (\Delta, 1))$.
%		\item[(b)] $\eta_i>1$, then    $(s'_i, \gamma'_i)=(\lb, (\Delta, \eta_i))$.
%       \end{enumerate}
\item [(iv)] 
	    If  $s_i=\lc$ and $\gamma_i=(\Delta, h_{i-1}+1)$,
	     then 
	     \begin{align}
		 (s_i', (\xi'_i, \eta_i'))=\begin{cases}
		 (\lc, (\Delta,1))& \text{if $\eta_i=1$}\\
		 (\lb, (\Delta, 1))& \text{if $\eta_i>1$}.
		 \end{cases}
		 \end{align}	    
		%$s_i'=\lc$ if $\eta_i=1$ and $s_i'=\lb$ if $\eta_i>1$, and $\eta_i'=1$. 
%	    there are two cases:
% 	   \begin{enumerate}
%	    \item[(a)] $h_{i-1}=0$,
%	    then  $(s'_i, \gamma'_i)=(\lc, (\Delta, 1))$.
%	    \item[(b)] $h_{i-1}>0$, 
%	    then  $(s'_i, \gamma'_i)=(\lb, (\Delta, 1))$.
%       \end{enumerate}
\item [(v)] 
    If  $s_i=\sd$ and $\gamma_i=(\xi_i, \eta_i)$, then $s_i'=\sd$. We use the 
    corresponding bipartite diagraphs $g_i$ and $g_i'$ to define 
    $\eta_i'$ and  $\xi_i'$.
   % Let us  define $\eta_i'$ and  $\xi_i')$.
			\begin{itemize}	
	\item Definition of $\eta'_i$. There are edges  $i\to v_1', v_1\to v_2', \ldots,v_{s-1}\to v_s'$
		in $g_i$ such that  $v_s$ is vacant or $v_s\to i'$.
        Assume that there are  $\xi^{*}_i-1$ vacant vertices at the left of 
         $v_s$  at the top row  
        of  $g_{i}$.   Let
        \begin{align}\label{V2-xi}
		\eta_i'=\begin{cases}
		1&\text{if $\xi_i=\xi_i^*$}\\
		\xi_i&\text{if $\xi_i>\xi_i^*$}\\
		\xi_i+1&\text{if $1\leq \xi_i<\xi_i^*$}.
		\end{cases}
		\end{align}
%		\begin{enumerate}
%		\item[(a)] $\xi_i=\xi^*_i,$ (i.e. $i'$ is connected to $v_s$ in $g_{i}$), 
%		then connect $i$ to the first vacant vertex of $\{1', \ldots, (i-1)'\}$ in $g'_{i-1}$.
%		Let  $\eta'_i=1$.
%		\item[(b)] $\xi_i>\xi^*_i$,
%		connect $i$ to $\xi_i$-th vacant vertex of $\{1', \ldots, (i-1)'\}$ in $g'_{i-1}$.
%		Let $\eta'_i=\xi_i$. 
%		\item[(c)] $\xi_i<\xi^*_i$,
%		connect $i$ to $(\xi_i+1)$-th vacant vertex of $\{1', \ldots, (i-1)'\}$ in $g'_{i-1}$.
%		Let $\eta'_i=\xi_i+1$.
%		\end{enumerate}
  	\item Definition of $\xi'_i$.  
	Connecting vertex $i$ to 
	the $\eta'_i$-th vacant vertex $p_1'$ at the bottom row of  $g'_{i-1}$ yeilds  edges  $i\to p_1', p_1\to p_2', \ldots,p_{t-1}\to p_t'$ such that $p_t$ is vacant in $g'_{i-1}$. 
	Assume that $p_t$ is 
	the $\eta^*_i$-th vacant vertex at the top row of $g'_{i-1}$. Let
	\begin{align}\label{V2-eta}
		\xi_i'=\begin{cases}
		\eta_i^*&\text{if $\eta_i=1$}\\
		\eta_i&\text{if $\eta_i>\eta_i^*$}\\
		\eta_i-1&\text{if $1<\eta_i\leq \eta_i^*$}.
		\end{cases}
		\end{align}
%\begin{enumerate}
%		\item[(a)] $\eta_i=1,$
%		then in $g'_{i-1}$
%		connect $i$ to $p_t'$.
%		Let $\xi'_i=\eta^*_i$.
%		\item[(b)] $\eta_i>\eta^*_i$,
%		then in $g'_{i-1}$
%		connect $i$ to $\eta_i$-th vacant vertex.
%		Let $\xi'_i=\eta_i$.
%		\item[(c)] $1<\eta_i\leq \eta^*_i$,
%		then in $g'_{i-1}$
%		connect $i$ to the $\eta_i-1$-th vacant vertex. 
%		Let $\xi'_i=\eta_i-1$.
%		\end{enumerate}
	\end{itemize}
	\end{itemize}
Set $\rho_2(s, \gamma)=(s', \gamma')$, where $s'=(s_1', \ldots, s_n')$ and
$\gamma'=(\gamma_1', \ldots, \gamma_n')$.

\begin{lem}
    The mapping $\rho_2: \L_n\to \L_n$ is an involution.
\end{lem}
\begin{proof} For  $(s,\gamma)\in \L_n$ let $\rho_2(s,\gamma)=(s',\gamma')$
	and
	$\rho_2(s',\gamma')=(s'',\gamma'')$.
	We show that  $(s_i'',\gamma_i'')=(s_i,\gamma_i)$
	by induction on $i$ with $1\leq i\leq n$.
	For $i=1$ we have  $(s_1,\gamma_1)=(\su, (\Delta, \Delta))$
	or $(\lc,(\Delta,1))$. It is clear from (1)--(iv) we have $(s_1'',\gamma_1'')=(s_1,\gamma_1)$. 
	Let $i\geq 2$ and 
	assume that  $((s_1'',\gamma_1''),\ldots,(s_{i-1}'',\gamma_{i-1}'')=((s_1,\gamma_1),\ldots,(s_{i-1},\gamma_{i-1}))
	$.
	For the cases (i)--(iv) it is easy to see that 
    $(s_i'',\gamma_i'')=(s_i, \gamma_i)$. 
    %%%%%%%%%%%%
    Here we just verify the case (v) with $s_i=D$ and $\gamma_i=(\sd,(\xi_i,\eta_i))$.	
    %Let $(s_i',\gamma_i')=\rho_2(s_i,\gamma_i)$. 
  Then $s''_i=s_i'=D$,
    and the mapping $\rho_2$ (resp. $\rho_2\circ \rho_2$) provides  the indices  $\xi^*_i$ and $\eta^*_i$ (resp.  $\tilde{\xi}^*_i$ and $\tilde{\eta}^*_i$)
    for  the computation of $\eta_i'$ and $\xi_i'$ (resp. $\eta_i''$ and $\xi_i''$)
    in \eqref{V2-xi} and \eqref{V2-eta}.
    We show that $\eta^*_i=\tilde{\xi}_i^*$, $\xi^*_i=\tilde{\eta}_i^*$,
    and $\gamma''_i=\gamma_i$ ($i\geq 2$) in the following way.

    Using $g_i$ and $g_i'$ we determine $\xi^*_i$  and $\eta^*_i$
    by (4.3) and (4.4), and then  characterize $i\in\Cyc(s,\gamma)$ and $i\in\Cyc(s',\gamma')$ 
    by the equation  $\xi_i=\xi^*_i$ 
   and  $\xi'_i=\eta^*_i$. This yields 
 $\gamma_i'=(\xi_i',\eta_i')$ by 
    (4.3) and (4.4).
    On the other hand,
    using $g_i'$ we get the index $\tilde{\xi}_i^*$ by (4.3)  and identify $i\in\Cyc(s',\gamma')$
    by the equation $\xi'_i=\tilde{\xi_i}^*$.
    Thus  $\eta^*_i=\tilde{\xi}_i^*$.				
    
    Next we define $\eta_i''$ by (4.3)
    using the index
    $\tilde{\xi}_i^*=\eta^*_i$.	
    \begin{align}
    \eta_i''=\begin{cases}
    1&\text{if $\xi'_i=\tilde{\xi}_i^*$}\\
    \xi'_i&\text{if $\xi'_i>\tilde{\xi}_i^*$}\\
    \xi'_i+1&\text{if $1\leq \xi'_i<\tilde{\xi}_i^*$}.
    \end{cases}
    \end{align}
    Comparing (4.5) with (4.4), we have $\eta_i''=\eta_i$.
    Now,  we know that $i\in \Cyc(s'',\gamma'')$ in $g_i''$
	if and only if  $\xi''_i=\tilde{\eta}_i^*$,
	and with $\eta''_i$  and $\gamma''_1,\ldots,\gamma''_{i-1}$ we can construct the edges 	
	$i\to v_1', v_1\to v_2', \ldots,v_{s-1}\to v_s'$ 
	in $g''_i$.	
	By induction hypothesis 
	$\gamma''_j=\gamma_j$ for $1\leq j\leq i-1$
    and $\eta''_i=\eta_i$.
	So 
    $\tilde{\eta}_i^*=\xi^*_i$.  Finally
   \begin{align}\label{V2-eta}
   \xi_i''=\begin{cases}
   \tilde{\eta}_i^*&\text{if $\eta'_i=1$}\\
   \eta'_i&\text{if $\eta'_i>\tilde{\eta}_i^*$}\\
   \eta'_i-1&\text{if $1<\eta'_i\leq \tilde{\eta}_i^*$}.
   \end{cases}
   \end{align}
Comparing  (4.6) with (4.3), we have $\xi_i''=\xi_i$.
Thus we complete the proof.	
	\end{proof}
\begin{lem} \label{rho2:properties} 
For $(s,\gamma)\in \L_n$, if $\rho_2(s,\gamma)=(s',\gamma')$, then
	%that For $\sigma\in \SS_n$ we have 
	\begin{equation*}
	(\Cyc ,\Arecp ,\Exc ,\Rar)(s',\gamma')
	=(\Arecp,\Cyc ,\Exc,\Rar)(s,\gamma) .  
	\end{equation*}	      
\end{lem}
\begin{proof} By definition of $\rho_2$ in \S~\ref{def:rho2}, if 
	$\rho_2(s_i,\gamma_i)=(s_i',\gamma_i')$ for $1\leq i\leq n$, then 
	$s_i=s_i'$ for  $s_i\in \{\su, \sd, \la\}$ and 
		$
		\{i: s_i=\lb\,\;\textrm{or}\,\;\lc\}=\{i: s'_i=\lb\,\;\textrm{or}\,\;\lc\}
		$.
	\begin{enumerate}
		\item 
		As $\Excp(s,\gamma)=\{i: s_i=\su\;\text{or}\; \la\}$ and 
		$\Excl(s,\gamma)=\{i: s_i=\sd\;\text{or}\; \la\}$, see \eqref{def:excpl}, we have 
		$\Exc (s',\gamma')=\Exc (s,\gamma)$.		
		\item %
		Recall that  $\Cyc(s,\gamma)=\{i: s_i=\lc \;\text{or}\; s_i=\sd \; \text{with}\; i\to i_{1}', i_{1}\to i_{2}', \ldots,i_{k}\to i'\; \text{in}\; g_i\}$ and $\Arecp (s',\gamma')=\{i: s'_i\in \{\sd,  \lb, \lc\}\; \text{and}\; \eta'_i=1\}$, see \eqref{def:arecp}-\eqref{def:cyc}. Now,
		\begin{itemize}
		    \item the case
			$(s_i,\gamma_i)=(\lc,(\Delta,h_{i-1}+1))$ means 
			 there is an edge $i\to i'$ in $g_i$, which is equivalent to  
			$(s'_i,\gamma'_i)=(\lb,(\Delta,1))$ 
			or $(\lc,(\Delta,1))$ by \S~\ref{def:rho2} (iv).
			\item
			the case $s_i=\sd$ with 
			$i\to v_1', v_1\to v_2', \ldots,v_{s-1}\to v_s',v_s\to i'$ in $g_i$ is equivalent to $s'_i=\sd$,
			$\eta'_i=1$
			by \S~\ref{def:rho2} (v).
		\end{itemize}
		Thus $\Arecp (s',\gamma')=\Cyc (s,\gamma)$.
		As $\rho_2$ is an involution from $\L_n$ onto itself, we 
		derive immediately that   $\Cyc(s',\gamma')=\Arecp(s,\gamma)$.
%		\begin{itemize}
%			\item 
%			$(s_i,\gamma_i)=(\lb,(\Delta,1))$ or $(\lc,(\Delta,1))$ 
%			if and only if $(s'_i,\gamma'_i)=(\lc,(\Delta,h_{i-1}+1))$, in other word there is an edge $i\to i'$ in $g_i'$ by \S~\ref{def:rho2} (iii) and (iv).
%			\item  $(s_i,\gamma_i)=(\sd,(\xi_i,1))$
%			if and only if
%			$s'_i=\sd$,
%			$\xi'_i=\xi_i$,
%			and in $g_i'$
%			there are edges $i\to i_{1}', i_{1}\to i_{2}', \ldots,i_{k}\to i'$ by 
%			\S~\ref{def:rho2} (v).
%		\end{itemize}
		\item Recall that 
		$\Rar(s,\gamma):=\{i: s_i=\lc \; \text{and}\;\eta_i=1 \}$.
		By  \S~\ref{def:rho2} (iv),
		$(s_i,\gamma_i)=(\lc,(\Delta,1))$
		if and only if $(s'_i,\gamma'_i)=(\lc,(\Delta,1))$. Hence  $\Rar (s',\gamma')=\Rar (s,\gamma)$.
		
	\end{enumerate}
\end{proof}

Let $\Phi:=\theta^{-1}\circ\rho_1\circ\theta$. 
By Lemmas~\ref{theta:properties} and \ref{rho2:properties}  we see that 
$\Phi$ is the desired bijection from $\SS_n$ onto itself for Theorem~2.

%
%By Lemmas~4 and 10  we see that 
%$\Phi(\sigma):=\theta^{-1}\circ\rho_2\circ\theta(\sigma)$
%is bijection from $\SS_n$ onto itself.

%%%%%%%%%%

\begin{figure}[t]
	\begin{tikzpicture}[scale=0.8]
	\node[above] at (1,0.5) {$1$};
	\node[above] at (2,0.5) {$2$};
	\node[above] at (3,0.5) {$3$};
	\node[above] at (4,0.5) {$4$};
	\node[above] at (5,0.5) {$5$};
	\node[above] at (6,0.5) {$6$};
	\node[above] at (7,0.5) {$7$};
	\node[above] at (8,0.5) {$8$};
	\node[above] at (9,0.5) {$9$};
	\node[above] at (10,0.5) {$10$};
	\node[above] at (11,0.5) {$11$};
	\node[above] at (12,0.5) {$12$};
	\node[above] at (13,0.5) {$13$};
	\node[above] at (14,0.5) {$14$};
	\node[above] at (15,0.5) {$15$};
	\node[above] at (16,0.5) {$16$};
	\node[above] at (17,0.5) {$17$};
	%%%%%%%%%%%%%%%%%%%%%%%
	\fill (1,0.5) circle (2.5pt);
	\fill (2,0.5) circle (2.5pt);
	\fill (3,0.5) circle (2.5pt);
	\fill (4,0.5) circle (2.5pt);
	\fill (5,0.5) circle (2.5pt);
	\fill (6,0.5) circle (2.5pt);
	\fill (7,0.5) circle (2.5pt);
	\fill (8,0.5) circle (2.5pt);
	\fill (9,0.5) circle (2.5pt);
	\fill (10,0.5) circle (2.5pt);
	\fill (11,0.5) circle (2.5pt);
	\fill (12,0.5) circle (2.5pt);
	\fill (13,0.5) circle (2.5pt);
	\fill (14,0.5) circle (2.5pt);
	\fill (15,0.5) circle (2.5pt);
	\fill (16,0.5) circle (2.5pt);
	\fill (17,0.5) circle (2.5pt);
	%%%%%%%%%%%%%%%%%%%%%%%%%%%%%5
	\fill (1,-1) circle (2.5pt);
	\fill (2,-1) circle (2.5pt);
	\fill (3,-1) circle (2.5pt);
	\fill (4,-1) circle (2.5pt);
	\fill (5,-1) circle (2.5pt);
	\fill (6,-1) circle (2.5pt);
	\fill (7,-1) circle (2.5pt);
	\fill (8,-1) circle (2.5pt);
	\fill (9,-1) circle (2.5pt);
	\fill (10,-1) circle (2.5pt);
	\fill (11,-1) circle (2.5pt);
	\fill (12,-1) circle (2.5pt);
	\fill (13,-1) circle (2.5pt);
	\fill (14,-1) circle (2.5pt);
	\fill (15,-1) circle (2.5pt);
	\fill (16,-1) circle (2.5pt);
	\fill (17,-1) circle (2.5pt);
	%%%%%%%%%%%%%%%%%%%%%%%%%%%%%%%
	\node[below] at (1,-1) {$1'$};
	\node[below] at (2,-1) {$2'$};
	\node[below] at (3,-1) {$3'$};
	\node[below] at (4,-1) {$4'$};
	\node[below] at (5,-1) {$5'$};
	\node[below] at (6,-1) {$6'$};
	\node[below] at (7,-1) {$7'$};
	\node[below] at (8,-1) {$8'$};
	\node[below] at (9,-1) {$9'$};
	\node[below] at (10,-1) {$10'$};
	\node[below] at (11,-1) {$11'$};
	\node[below] at (12,-1) {$12'$};
	\node[below] at (13,-1) {$13'$};
	\node[below] at (14,-1) {$14'$};
	\node[below] at (15,-1) {$15'$};
	\node[below] at (16,-1) {$16'$};
	\node[below] at (17,-1) {$17'$};
	%%%%%%%%%%%%%%%%%%%%%%%55
	\draw [>=stealth,  thick,  ->](1,0.5)--(4,-1);
	\draw [>=stealth,  thick,  ->](2,0.5)--(11,-1);
	\draw [>=stealth,  thick,  ->](3,0.5)--(2,-1);
	\draw [>=stealth,  thick,  ->](4,0.5)--(9,-1);
	\draw [>=stealth,  thick,  ->](5,0.5)--(1,-1);
	\draw [>=stealth,  thick,  ->](6,0.5)--(10,-1);
	\draw [>=stealth,  thick,  ->](7,0.5)--(7,-1);
	\draw [>=stealth,  thick,  ->](8,0.5)--(8,-1);
	\draw [>=stealth,  thick,  ->](9,0.5)--(5,-1);
	\draw [>=stealth,  thick,  ->](10,0.5)--(3,-1);
	\draw [>=stealth,  thick,  ->](11,0.5)--(6,-1);
	\draw [>=stealth,  thick,  ->] (12,0.5)--(12,-1);
	\draw [>=stealth,  thick,  ->](13,0.5)--(17,-1);
	\draw [>=stealth,  thick,  ->](14,0.5)--(16,-1);
	\draw [>=stealth,  thick,  ->](15,0.5)--(15,-1);
	\draw [>=stealth,  thick,  ->](16,0.5)--(14,-1);
	\draw [>=stealth,  thick,  ->][->] (17,0.5)--(13,-1);
	\end{tikzpicture}
	\caption{\small  The  graph of $\Phi(\sigma)=(1,4,9,5)(7)(8)(2,11,6,10,3)(12)(13,17)(14,16)(15)$ }\label{graph-g2}
\end{figure}
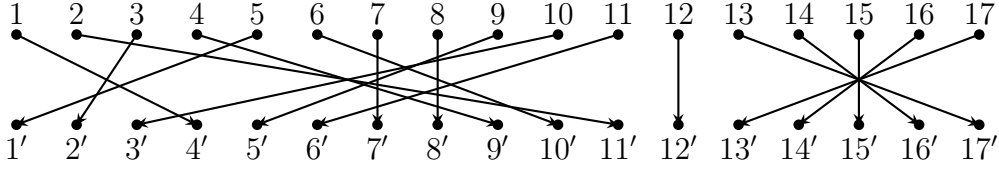

%%%
\begin{figure}[t]
	\begin{center}
		\begin{tikzpicture}[scale=0.9]
		\draw[step=1cm, gray, thick,dotted] (0, 0) grid (17,4);
		
		\draw[black] (0, 0)--(1, 1)--(2, 2)--(3, 2)--(4, 2)--(5, 2)--(6, 3)--(7, 3)--(8, 3)--(9, 2)--(10, 1)--(11, 0)--(12, 0)--(13, 1)--(14, 2)--(15, 2)--(16,1)--(17,0);
		\draw[black] (0,0) node {$\bullet$};
		\draw[black] (1,1) node {$\bullet$};
		\draw[black] (2,2) node {$\bullet$};
		\draw[black] (3,2) node {$\bullet$};
		\draw[black] (4,2) node {$\bullet$};
		\draw[black] (5,2) node {$\bullet$};
		\draw[black] (6,3) node {$\bullet$};
		\draw[black] (7,3) node {$\bullet$};
		\draw[black] (8,3) node {$\bullet$};
		\draw[black] (9,2) node {$\bullet$};
		\draw[black] (10,1) node {$\bullet$};
		\draw[black] (11,0) node {$\bullet$};
		\draw[black] (12,0) node {$\bullet$};
		\draw[black] (13,1) node {$\bullet$};
		\draw[black] (14,2) node {$\bullet$};
		\draw[black] (15,2) node {$\bullet$};
		\draw[black] (16,1) node {$\bullet$};
		\draw[black] (17,0) node {$\bullet$};
		
		\draw [-,thin, blue](3,2)--(4,2);
		\draw [-,thin, black](2,2)--(3,2);
		\draw [-,thin, red](6,3)--(7,3);
		\draw [-,thin, red](7,3)--(8,3);
		\draw [-,thin, red](14,2)--(15,2);
		\draw [-,thin, black](4,2)--(5,2);
		\draw [-,thin, red](11,0)--(12,0);
		
		\tiny{
			\draw[black] (0.5,-0.5) node {$(\Delta,\Delta)$};
			\draw[black] (1.5,-0.5) node {$(\Delta,\Delta)$};
			\draw[black] (2.5,-0.5) node {$(\Delta,2)$};
			\draw[black] (3.5,-0.5) node {$(1,\Delta)$};
			\draw[black] (4.5,-0.5) node {$(\Delta,1)$};
			\draw[black] (5.5,-0.5) node {$(\Delta,\Delta)$};
			\draw[black] (6.5,-0.5) node {$(\Delta,4)$};
			\draw[black] (7.5,-0.5) node {$(\Delta,4)$};
			\draw[black] (8.5,-0.5) node {$(2,2)$};
			\draw[black] (9.5,-0.5) node {$(2,1)$};
			\draw[black] (10.5,-0.5) node {$(1,1)$};
			\draw[black] (11.5,-0.5) node {$(\Delta,1)$};
			\draw[black] (12.5,-0.5) node {$(\Delta,\Delta)$};
			\draw[black] (13.5,-0.5) node {$(\Delta,\Delta)$};
			\draw[black] (14.5,-0.5) node {$(\Delta,3)$};
			\draw[black] (15.5,-0.5) node {$(2,2)$};
			\draw[black] (16.5,-0.5) node {$(1,1)$};
			
			\draw[black] (2.5,2.5) node {$L_b$};
			\draw[black] (3.5,2.5) node {$L_a$};
			\draw[black] (4.5,2.5) node {$L_b$};
			\draw[black] (6.5,3.5) node {$L_c$};
			\draw[black] (7.5,3.5) node {$L_c$};
			\draw[black] (11.5,0.5) node {$L_c$};
			\draw[black] (14.5,2.5) node {$L_c$};
		}	
		\end{tikzpicture}
		\caption{\small The Laguerre history of  $\Phi(\sigma)$}
		%)=(1,4,9,5)(2,10,6,11,3)(7)(8)(12)(13,17)(14,16)(15)$}
	\end{center}
\end{figure}
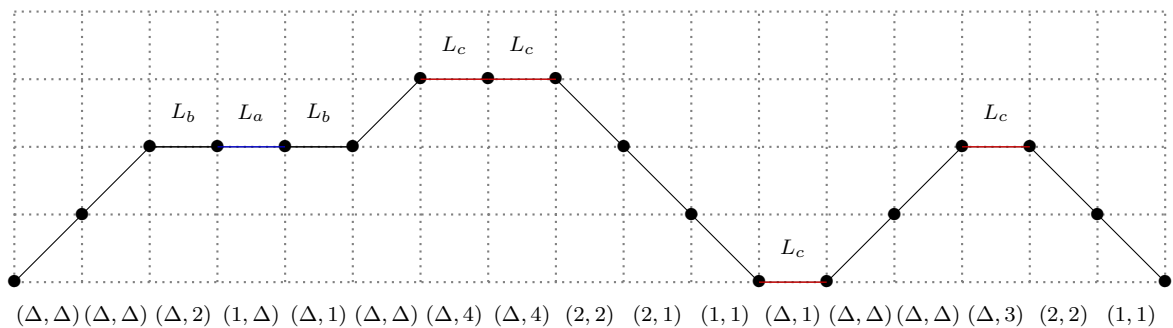  
            %%%%%

\begin{ex}  
For our running example $\sigma$,
we get  $\tau:=\Phi(\sigma)$ with
$$\tau=\left(%
\begin{array}{cccccccccccccccccc}
1 & 2 & 3 & 4 & 5 & 6 & 7 & 8 & 9 & 10 & 11 & 12 & 13 & 14 & 15 & 16 & 17 \\
4 & 11 & 2 & 9 & 1 & 10 & 7 & 8 & 5 & 3& 6 & 12 & 17 & 16 & 15 & 14 & 13  \\
\end{array}%
\right).$$
Thus
\begin{align*}
\Rec\tau&=(\{1, \,2, \,6, \,12, \,13\},\,
\{4, \,10, \,11, \,12, \,17\})\\
\Arec\tau&=(\{
5, \,10, \,11, \,12, \,17\},\,
\{1, \,3, \,6, \,12, \,13\})\\
\Erec\tau&=(\{1, \,2, \,6, \,13\},\,
\{4, \,10, \,11, \,17\})\\
\Exc\tau&=(\{
1, \,2, \,4, \,6, \,13, \,14\},\, 
\{4, \,9, \,10, \,11, \,16, \,17\})
\end{align*}
and $\Rar\tau=\{12\}$.
Also, the cycle decomposition of $\tau$   is 
\begin{equation*}
\tau=(1,4,9,5)\,\,(2,11, 6,10,3)\, \,(7)\,(8)\,(12)\,(13,17)\,(14,16)\,(15).
\end{equation*}
Thus $\Fix\tau=\{7,8,12,15\}$, 
$\Cyc\tau=\{7,8,9,11,12,15,16, 17\}$, 
$\Cpeak\tau=\{9,10,11,16,17\}$,
$\Cval\tau=\{1,2,6,13,14\}$,
$\Cdrise\tau=\{4\}$, and 
$\Cdfall\tau=\{3,5\}$.

%The associated graphs  are  depicted  in Figure~\ref{fig6}.

\end{ex}

\section{Acknowledgements}
The first author was 
supported by the \emph{China Scholarship Council}.
This work was done during  
his  visit  at  Universit\'e
Claude Bernard Lyon 1 in 2018-2019.

\end{document}